\documentclass[11pt]{article}
\usepackage{amsfonts}
\usepackage{amssymb}
\usepackage{amsthm}
\usepackage{amsmath}
\usepackage{graphicx}
\usepackage{empheq}
\usepackage{indentfirst}
\usepackage{cite}
\usepackage{mathrsfs}
\usepackage{graphics}
\usepackage{enumerate}
\usepackage{color}

 \newtheorem{theorem}{Theorem}[section]
 
 \newtheorem{lemma}{Lemma}[section]
 \newtheorem{proposition}{Proposition}[section]

 \newtheorem{remark}{Remark}[section]
 \numberwithin{equation}{section}

\allowdisplaybreaks[4]

\newcommand{\nn}{\nonumber}
\newcommand{\io}{\int_\Omega}
\newcommand{\R}{\mathbb{R}}

\newcommand{\beq}{\begin{equation}}
\newcommand{\eeq}{\end{equation}}
 \def\non{\nonumber }
\def\bea{\begin{eqnarray}}
\def\eea{\end{eqnarray}}

\begin{document}
\title{Global Existence for a Kinetic Model of Pattern Formation with Density-suppressed Motilities}
\author{Kentarou Fujie\thanks{Research Alliance Center for Mathematical Sciences, Tohuku University, Sendai 980-8578, Miyagi, Japan, \textsl{fujie@tohoku.ac.jp}},
	\ Jie Jiang\thanks{Wuhan Institute of Physics and Mathematics, Chinese Academy of Sciences,
		Wuhan 430071, HuBei Province, P.R. China,
		\textsl{jiang@wipm.ac.cn}.}}

\date{\today}

\maketitle

\begin{abstract} In this paper, we consider global existence of classical solutions to the following kinetic model of pattern formation 
	\begin{equation}
	\begin{cases}\label{chemo0}
	u_t=\Delta (\gamma (v)u)+\mu u(1-u)\\
	-\Delta v+v=u
	\end{cases}
	\end{equation}in a smooth bounded  domain $\Omega\subset\mathbb{R}^n$, $n\geq1$ with no-flux boundary conditions. Here,  $\mu\geq0$ is any given constant. The  function $\gamma(\cdot)$ represents a signal-dependent diffusion motility and is decreasing in $v$ which models a density-suppressed motility in process of stripe pattern formation through self-trapping mechanism \cite{PRL12,Sciencs11}.
	
	The major difficulty in analysis lies in the possible degeneracy of diffusion as $v\nearrow+\infty.$ In the present contribution, based on  a subtle observation of the nonlinear structure, we develop a new method to rule out finite-time degeneracy in any spatial dimension for all smooth motility function satisfying $\gamma(v)>0$ and $\gamma'(v)\leq0$ for $v\geq0$. Then we prove global existence of classical solution for \eqref{chemo0} in the two-dimensional setting  with  any $\mu\geq0$. Moreover, the global solution is proven to be uniform-in-time bounded if either $1/\gamma$ satisfies certain polynomial growth condition or $\mu>0.$ 
	
	Besides, we pay particular attention to the specific case $\gamma(v)=e^{-v}$ with $\mu=0$. Under the circumstances, system \eqref{chemo0} becomes of great interest because it shares the same set of equilibria as well as the Lyapunov functional with the classical Keller--Segel model. A novel critical phenomenon in the two-dimensional setting is observed that with any initial datum of sub-critical mass, the global solution is proved to be uniform-in-time bounded, while with certain initial datum of super-critical mass, the global solution will become unbounded as time goes to infinity. Namely, blowup takes place in infinite time rather than finite time in our model which is distinct from the well-known fact that certain initial data of super-critical mass will enforce a finite-time blowup for the classical Keller--Segel system.

\end{abstract}
\section{Introduction}
Experimental observations show that colonies of bacteria and simple eukaryotes can generate complex shapes and patterns. In order to understand the mechanism of pattern formation, extensive mathematical models were derived including the Keller--Segel system modeling the pattern formation driven by chemotactic bacteria. In most cases,  the models invoke nonlinear diffusion of the cells where the diffusion coefficient increases with the local density \cite{BBTW15}.

Recently, it was theoretically proposed  in \cite{Sciencs11,PRL12} that density-suppressed motility could also lead to patterns via the so-called ''self-trapping" mechanism.  A  kinetic model with signal-dependent motility was proposed to describe the processing of stripe pattern formation through self-trapping. Denoting the cell density by $u$ and chemical concentration by $v$, the diffusion and production of $v$ is governed by 
\begin{equation}\label{v1}
	\varepsilon v_t-\Delta v+v=u,
\end{equation}
while the stochastic swim-and-tumble motion of cells  is modeled by the following diffusion equation with a logistic growth:
\begin{equation}\label{u1}
	u_t=\Delta(\gamma(v)u)+\mu u(1-u).
\end{equation}
Here $\mu\geq0$ and the motility function $\gamma(\cdot)>0$ depends explicitly on $v$. Moreover, for all $v>0$, it is assumed that
\begin{equation}\label{dec}
	\gamma'(v)<0,
\end{equation}
 since it takes into account the repressive effect of signal concentration (and hence cell density) on cell motility. Note that $\gamma(v)$  may approach to zero as $v\nearrow+\infty$, which characterizes the incessant tumbling of cells at high concentration, resulting in a vanishing macroscopic motility. Simulation results of \cite{Sciencs11,PRL12} show that this model correctly captures the dynamics at propagating front where new stripes are formed.
 
To the best of our knowledge, there are only a few theoretical results on this kinetic model in the literature. An essential difficulty in analysis lies in the possible degeneracy of diffusion in \eqref{u1} as $v\nearrow+\infty.$ By assuming uniform upper and lower boundedness for $\gamma$ as well as its derivative, Tao and Winkler \cite{TaoWin17} studied the  fully parabolic system consisting of \eqref{v1}-\eqref{u1} with $\mu=0$ under Neumann boundary conditions, where existence of global classical solutions in two dimensions and global weak solutions in the three dimensions were established. However, degeneracy was prevented due to their technical assumptions on $\gamma.$

Meanwhile, results on global existence with degenerate motility are rather limited. When $\mu>0$ and in the two-dimensional setting, Jin et al \cite{JKW18} proved existence of globally bounded classical solutions to the fully parabolic system permitting a general kind of degenerate motility functions. Moreover, they obtained convergence toward constant steady states provided that $\mu>\mu_*$ with some $\mu_*>0$ depending on $\gamma$.  However, their assumption on $\gamma(v)$ excluded very fast decay functions such as $e^{-v^2}$ or $e^{-e^{v}}$.

From a mathematical point of view, the problem becomes more challenging when $\mu=0$. In this case, to the authors' knowledge, only the following specific polynomial decay function was considered in the literature, i.e.,
\begin{equation*}
	\gamma(v)=\frac{c_0}{v^{k}},
\end{equation*}
with some $c_0,k>0$. Yoon and Kim \cite{YK17} investigated the initial-Neumann boundary problem where global existence was obtained for any $k>0$ under a smallness assumption on $c_0$. The only global existence result without smallness assumptions was recently given by Ahn and Yoon \cite{Anh19}. They considered the simplified parabolic-elliptic version of above system, that is
\begin{equation*}
	\begin{cases}
	u_t=\Delta (v^{-k}u)&x\in\Omega,\;t>0\\
	-\Delta v+v=u&x\in\Omega,\;t>0
	\end{cases}
\end{equation*}with homogeneous Neumann boundary conditions.  They established global existence of classical solutions with a uniform-in-time bound when $n\leq2$ for any $k>0$ or $n\geq3$ for $k<\frac{2}{n-2}$.

In the present contribution, we consider the initial-Neumann boundary value problem of the following parabolic-elliptic system:
\begin{equation}
\begin{cases}\label{chemo1}
u_t=\Delta (\gamma (v)u)+\mu u(1-u)&x\in\Omega,\;t>0\\
-\Delta v+v=u&x\in\Omega,\;t>0\\
\partial_\nu u=\partial_\nu v=0,\qquad &x\in\partial\Omega,\;t>0\\
u(x,0)=u_0(x),\qquad & x\in\Omega
\end{cases}
\end{equation}
where $\Omega\subset\mathbb{R}^n$ with $n\geq1$  is a smooth bounded domain. We study global existence of classical solutions to the above problem for general type of degenerate motility functions with any $\mu\geq0$. To be more precisely, we assume  throughout this paper that
\begin{equation}\label{ini}
u_0\in C^0(\overline\Omega),\quad u_0\geq0 \quad  \mbox{in } \overline\Omega, \quad u\not\equiv0
\end{equation}and for $\gamma$, we require that
\begin{equation}\label{gamma}
\mathrm{(A1)}:\qquad\gamma(v)\in C^3[0,+\infty),\;\gamma(v)>0\;\text{and}\;\gamma'(v)\leq0\;\text{on}\;(0,+\infty).
\end{equation}

As we mentioned above, the main obstacle in analysis comes from the possible degeneracy as $v\nearrow+\infty.$ Thus, in order to rule out degeneracy, one needs to obtain an upper bound for $v$. The typical way  is to derive the $L^\infty(0,T;L^p(\Omega))$ boundedness of $u$ for any $p>\frac{n}{2}$ which will directly yield the $L^\infty(0,T;L^\infty(\Omega))$ boundedness of $v$ according to the second equation in \eqref{chemo1} as done in previous studies \cite{Anh19,YK17} or in related work on the classical Keller--Segel models \cite{BBTW15}. However, this method fails for our system with general motility functions since in most cases, non-degeneracy is requisite to get the $L^\infty(0,T;L^p(\Omega))$ boundedness of $u$.

The novelty of the present paper is that a new approach is introduced to rule out finite-time degeneracy for all monotone decreasing motility functions directly. Based on a subtle observation, we find that the following identity
\begin{equation}\label{keyid}
	\partial_tv(x,t)+u\gamma(v)+\mu (I-\Delta)^{-1}[u^2]=(I-\Delta)^{-1}[u\gamma(v)+\mu u](x,t)
\end{equation}
holds for any smooth solution $(u,v)$  which unveils the hidden mechanism of the special structure of system \eqref{chemo1} and is the key ingredient to prove global existence.  Here, $(I-\Delta)^{-1}$ denotes the inverse operator of $I-\Delta$ and $\Delta$ is the usual Laplacian operator with homogeneous Neumann boundary condition. Several tricks are developed along with this key identity to derive upper bounds of $v$ under different circumstances. Roughly speaking, in view of the positivity of $u\gamma(v)+\mu(I-\Delta)^{-1} [u^2]$, thanks to the uniform-in-time lower boundedness $\inf_{x\in\Omega}v(x,t)\geq v_*$ given by Lemma \ref{lowerbound} below and the decreasing property of $\gamma(\cdot)$, we can deduce by the comparison principle together with Gronwall's inequality that
\begin{equation}\label{upb}
	v(x, t)\leq v_0(x)e^{(\gamma(v_*)+\mu)t}
\end{equation}  for any $(x,t)\in \Omega \times[0,T_{\mathrm{max}})$ with $v_0\triangleq(I-\Delta)^{-1}u_0$ and $T_{\mathrm{max}}$ being the maximal time of existence of classical solutions. Thus,  finite-time degeneracy cannot take place and hence global existence can be investigated by the classical energy method as done for Keller--Segel systems. On the other hand, with any $\mu>0$ if $n\leq3$ or any $\mu>\gamma(v_*)$ if $n\geq4$, we can further prove that the point-wise upper bound of $v(x,t)$ is in fact time-independent and hence uniform-in-time boundedness of the solutions can be discussed under the circumstances. In the present paper, we focus on the two-dimensional case and  the higher dimensional problem will be studied in our future works.

Now we are in a position to state our first main result on global existence of classical solutions with general motility functions in two dimensions.
\begin{theorem}\label{TH1}
	Assume $\Omega\subset\mathbb{R}^2$ and $\gamma(\cdot)$ satisfies (A1). For any given initial datum $u_0$ satisfying \eqref{ini}, system \eqref{chemo1} permits a unique classical solution $(u,v)\in (C^0(\overline{\Omega}\times[0,\infty))\cap C^{2,1}(\overline{\Omega}\times(0,\infty)))^2$. If $\mu>0$, then $(u,v)$ is uniform-in-time bounded in the sense that
	\begin{equation*}
		\|u(\cdot,t)\|_{L^\infty(\Omega)}+\|v(\cdot,t)\|_{L^\infty(\Omega)}\leq C\quad\text{for all}\;t>0
	\end{equation*}with some $C>0$ depending on $u_0$ and $\Omega$ only. 
	
	 Moreover, if
	\begin{equation}\label{gam2}
	K_0\triangleq\max\limits_{0\leq s\leq +\infty}\frac{|\gamma'(s)|^2}{\gamma(s)}<+\infty
	\end{equation}
	and $\mu>\frac{K_0}{16}$, there holds
	\begin{equation*}
	\lim\limits_{t\rightarrow+\infty}\left(\|u(\cdot,t)-1\|_{L^\infty(\Omega)}+\|v(\cdot,t)-1\|_{L^\infty(\Omega)}\right)=0.
	\end{equation*}
\end{theorem}
\begin{remark}
For global existence and uniform-in-time boundedness of classical solutions, we do not need existence of  $\lim\limits_{v\rightarrow+\infty}\frac{\gamma'(v)}{\gamma(v)}$ as required in \cite{JKW18}. Thus, fast decay motilities such as $\gamma(v)=e^{-v^2}$ or $\gamma(v)=e^{-e^{v}}$ are permitted in our case for global existence. 
\end{remark}
We would like to mention that in general the global solution when $\mu=0$ may become unbounded as time goes to infinity since the upper bound of $v$ grows in time due to \eqref{upb}. However, if we propose the  following additional growth condition on $\gamma(\cdot)$: 
\begin{equation}\label{gamma2}\mathrm{(A2)}:\qquad\text{there is $k>0$ such that}
	\lim\limits_{s\rightarrow+\infty}s^{k}\gamma(s)=+\infty,
\end{equation}
we can also prove the uniform-in-time boundedness of the global solution as follows.
\begin{theorem}\label{TH3}
	Assume $n=2$, $\mu=0$ and $\gamma(\cdot)$ satisfies (A1) and (A2). Then the global classical solution is uniform-in-time bounded.
\end{theorem}
\begin{remark}\label{singularity0}With the aid of Lemma \ref{lowerbound}, $v$ is bounded from below by a strictly positive constant $v_*$ when $\mu=0$. Thus, we may allow $\gamma(s)$ to has singularity at $s=0$ since we can simply replace $\gamma(s)$ by a new motility function $\tilde{\gamma}(s)$ which satisfies $\mathrm{(A1)}$ and coincides with $\gamma(s)$ for $s\geq\frac{v_*}{2}$. Thus, Theorem \ref{TH3} generalizes the existence result in 2D of \cite{Anh19} for $\gamma(v)=v^{-k}$ to more general motility functions satisfying (A2), for example, $\gamma(v)=\frac{1}{v^{k}\log(1+v)}$ with any $k>0$.
\end{remark}

Next, we consider system \eqref{chemo1} with the specific motility function $\gamma(v)=e^{-v}$ and $\mu=0$ in the two-dimensional case, that is,
\begin{equation}\label{chemo2a}
	\begin{cases}
		u_t=\Delta (ue^{-v})=\nabla \cdot(e^{-v}(\nabla u-u\nabla v)),&x\in\Omega,\;t>0\\
		-\Delta v+v=u,&x\in\Omega,\;t>0.
	\end{cases}
\end{equation}

System \eqref{chemo2a} is of great interest because it resembles the classical parabolic-elliptic Keller--Segel system:
\begin{equation}\label{ks}
	\begin{cases}
	u_t=\nabla\cdot(\nabla u-u\nabla v)\\
	-\Delta v+v=u\\
	\partial_\nu u=\partial_\nu v=0.
	\end{cases}
\end{equation}
Besides, they share certain important features. First, they have the same stationary problem which reads
\begin{equation*}
\begin{cases}
-\Delta v+v=\Lambda e^{v}/\int_\Omega e^{v}\,dx\;\;\text{in}\;\Omega\\
u=\Lambda e^{v}/\int_\Omega e^{v}\,dx\;\;\text{in}\;\Omega\\
\partial_\nu v=0\;\;\text{on}\;\partial\Omega
\end{cases}
\end{equation*}with $\Lambda=\|u_0\|_{L^1(\Omega)}>0$.

Second, they have the same Lyapunov functional. Indeed, for any smooth solution $(u,v)$ of \eqref{chemo2a}, the following energy-dissipation relation holds
\begin{equation}\label{Lyapunov1}
\frac{d}{dt}E(u,v)(t)+\int_\Omega ue^{-v}\left|\nabla \log u-\nabla v\right|^2dx=0,
\end{equation} where the Lyapunov functional is defined by
\begin{equation*}
E(u,v)=\int_\Omega \left(u\log u+\frac12|\nabla v|^2+\frac12 v^2-uv\right)dx.
\end{equation*} In comparison, for system \eqref{ks}, there holds
\begin{equation*}
	\frac{d}{dt}E(u,v)(t)+\int_\Omega u\left|\nabla \log u-\nabla v\right|^2dx=0.
\end{equation*} The only difference lies in an extra weighed function $e^{-v}$ appearing in the dissipation term in \eqref{Lyapunov1}.

Therefore, an interesting question is whether the behaviors of solutions to \eqref{chemo2a} and \eqref{ks} are similar. A well-known fact of the Keller--Segel model \eqref{ks} is that classical solutions with large initial data may blow up when dimension $n\geq2$ (see, e.g., \cite{HW01,Nagai98, Nagai00,nagai2001}), i.e., there is $T_{\mathrm{max}}\in(0,+\infty]$ such that
\begin{equation*}
\lim\limits_{t\nearrow T_{\mathrm{max}}}\left(	\|u(\cdot,t)\|_{L^\infty(\Omega)}+\|v(\cdot,t)\|_{L^\infty(\Omega)}\right)=+\infty.
\end{equation*}
In particular, a critical-mass phenomenon exists in the two-dimensional case. If the conserved total mass of cells $\Lambda\triangleq\int_\Omega u_0 dx$ is lower than certain threshold number $\Lambda_c$, then global classical solution exists and remains bounded for all times\cite{Nagai97}; otherwise, it may blow up in finite or infinite time\cite{HW01,ssMAA2001}. Existence of finite-time blowup was examined in \cite{nagai2001,Win13,jaeger_luckhaus}. However, to our knowledge, infinite-time blowup was only obtained for Cauchy problem when the second equation of \eqref{ks} is replaced by $-\Delta v=u$ in \cite{BCM10,GM18} with critical mass $8\pi$.

In contrast, we also observe an interesting critical phenomenon for system \eqref{chemo2a} in the two-dimensional setting. Classical solution exists globally for any initial datum with arbitrarily large total mass by Theorem \ref{TH1} which means no finite-time blowup occurs. Moreover, the solution is uniform-in-time bounded if the total mass is less than $\Lambda_c$, while  with certain initial datum of supper-critical mass one can construct global classical solution which blows up at time infinity, i.e.,
\begin{equation}
\lim\limits_{t\nearrow +\infty}\left(	\|u(\cdot,t)\|_{L^\infty(\Omega)}+\|v(\cdot,t)\|_{L^\infty(\Omega)}\right)=+\infty.\nn
\end{equation}
More precisely, we obtain the result for problem \eqref{chemo2a} as follow.
\begin{theorem}\label{TH2}
Assume $n=2$, $\gamma(v)=e^{-v}$, $\mu=0$ and $u_0$ satisfies \eqref{ini}. Let
\begin{equation}
\Lambda_c=\begin{cases}
8\pi\qquad\text{if}\;\Omega=B_R(0)\triangleq\{x\in\mathbb{R}^2;\;|x|<R\}\;\;\text{with some}\;0<R<\infty\;\text{and}\;u_0 \;\text{is radial in}\;x\\
4\pi\qquad\text{otherwise.}		\nn
\end{cases}
\end{equation} Then if $\Lambda\triangleq\int_\Omega u_0 dx<\Lambda_c$, the global classical solution of \eqref{chemo2a} is uniform-in-time bounded. Moreover, the solution converges to an equilibrium as time goes to infinity.

  On the other hand, if $\Omega=B_R(0)$, there exists radially symmetric initial datum $u_0$ with $\Lambda\in(8\pi,\infty)\backslash4\pi\mathbb{N}$ such that the corresponding global classical solution blows up at time infinity. More precisely,
\begin{equation*}
\lim\limits_{t\nearrow +\infty}\|u(\cdot,t)\|_{L^\infty(\Omega)}=\lim\limits_{t\nearrow +\infty}\int_\Omega uv dx=\lim\limits_{t\nearrow +\infty}\int_\Omega (|\nabla v|^2+v^2)dx=\lim\limits_{t\nearrow +\infty}\int_\Omega e^{\alpha v}dx=+\infty
\end{equation*}for any $\alpha>\frac12.$
\end{theorem}

\begin{remark}
Similar mass critical phenomenon as mentioned above was established for chemotaxis models in some special cases \cite{CS, TaoWincritical}.
\end{remark}

Uniform-in-time boundedness of $v$ with sub-critical mass is somehow tricky as it is for the case in Theorem \ref{TH3}. Classical iteration approach \cite{Ali} fails in our case where uniform-in-time $\|v(t,\cdot)\|_{L^\infty}$ boundedness is obtained by proving $\|v(t,\cdot)\|_{L^p}\leq C$ for any $p>1$ with some $C>0$ independent of $p$ and $t$. In this paper, based on some delicate estimates and the the classical result in \cite{Nagai97}, utilizing the key identity \eqref{keyid} and the uniform Gronwall inequality, we develop a new method to establish the uniform-in-time point-wise upper bound of $v$.
 
 In view of the same features of system \eqref{chemo2a} shared with the Keller--Segel system \eqref{ks} mentioned above,  infinite-time blowup is proven following the idea in \cite{HW01,ssMAA2001} since on the one hand, it was shown that if $\Lambda\notin4\pi\mathbb{N}$, then for any initial data $u_0$ emanating a uniform-in-time global solution, $E(u_0,v_0)$ with $v_0=(I-\Delta)^{-1} u_0$ must be bounded from below. On the other hand, we may construct a sequence of initial data $(u_{0\lambda},v_{0\lambda})$ with $v_{0\lambda}=(I-\Delta)^{-1}u_{0\lambda}$ such that $E(u_{0\lambda},v_{0\lambda})\rightarrow-\infty$ as $\lambda\rightarrow+\infty$. Therefore, the global solution starting from $(u_{0\lambda},v_{0\lambda})$ must blow up in infinite time.

Existence of such initial data $(u_{0\lambda},v_{0\lambda})$ was proved for the fully parabolic Keller--Segel system \cite{HW01}. However in the parabolic-elliptic case, there is an addition constraint on the initial data that $v_{0\lambda}-\Delta v_{0\lambda}=u_{0\lambda}$ in $\Omega$ and $\partial_\nu v_{0\lambda}=0$ on $\partial\Omega$ which means $u_{0\lambda}$ and $v_{0\lambda}$ cannot be independently chosen as in \cite{HW01}. Thus it provides us with much less freedom for the construction in the latter case. Existence of such kind of initial data in the radially symmetric case was claimed in \cite{ssMAA2001} with no detail. Similar problem was tackled for quasilinear parabolic-elliptic Keller--Segel systems recently in higher dimensions in \cite{Lankeit}. However their construction fails in our case. Since the authors find no references providing us the desired construction, we give a concrete example in detail in Section \ref{construction}.
We remark that our construction can also be applied to the parabolic-elliptic Keller-Segel system.

 The rest of the paper is organized as follows. In Section 2, we provide some preliminary results and recall some useful lemmas. Then we prove the key identity \eqref{keyid} in Section 3 and establish point-wise upper bounds for $v$ in various situations. Thanks to the upper bound of $v$, we are able to study global existence of classical solutions in Section 4. The last section is devoted to  the case $\gamma(v)=e^{-v}$ and $\mu=0$ where the new critical phenomenon is proved in the two-dimensional setting.
 
\section{Preliminaries}
In this section, we recall some lemmas which will be used in the sequel. First, local existence and uniqueness of classical solutions to system \eqref{chemo1} can be
established by the standard fixed point argument and  regularity theory for elliptic equations. Similar proof can be found in \cite[Lemma 3.1]{Anh19} or \cite[Lemma 2.1]{JKW18}  and hence here we omit the detail here.
\begin{theorem}\label{local}
	Let $\Omega$ be a smooth bounded domain of $\mathbb{R}^n$. Suppose that $\gamma(\cdot)$ satisfies \eqref{gamma} and $u_0$ satisfies \eqref{ini}. Then there exists $T_{\mathrm{max}} \in (0, \infty]$ such that problem \eqref{chemo1} permits a unique classical solution $(u,v)\in (C^0(\overline{\Omega}\times[0,T_{\mathrm{max}}))\cap C^{2,1}(\overline{\Omega}\times(0,T_{\mathrm{max}})))^2$. 

	 If $T_{\mathrm{max}}<\infty$, then
	\begin{equation*}
	\lim\limits_{t\nearrow T_{\mathrm{max}}}\|u(\cdot,t)\|_{L^\infty(\Omega)}=\infty.
	\end{equation*}
	
	Moreover, the solution $(u,v)$ satisfies the mass conservation when $\mu=0:$
	\begin{equation*}
	\int_{\Omega}u(\cdot,t)dx=\int_\Omega v(\cdot,t)dx=\int_{\Omega}u_0 dx
	\quad \text{for\ all}\ t \in (0,T_{\mathrm{max}}).
	\end{equation*}	
\end{theorem}
Note that a strictly positive uniform-in-time lower bound for $v$ is given in \cite[Corollary 2.3]{Anh19}.
\begin{lemma}\label{lowerbound}
	Suppose $\mu=0$ and $(u,v)$  is the classical solution of \eqref{chemo1} up to the maximal time of existence $T_{\mathrm{max}}\in(0,\infty]$. Then, there exists a strictly positive constant $v_*=v_*(n,\Omega,\|u_0\|_{L^1(\Omega)})$ such that for all $t\in(0,T_{\mathrm{max}})$, there holds
	\begin{equation*}
		\inf\limits_{x\in\Omega}v(x,t)\geq v_*.
	\end{equation*}
\end{lemma}
\begin{remark}
	If $\mu>0$, $\|u\|_{L^1(\Omega)}$ is not conserved and may decay in time. Under the circumstances, we cannot obtain a strictly positive \textbf{ time-independent} lower bound for $v$. In other word, $v_*>0$ depends on time in general if $\mu>0$ (cf. \cite{FWY2014}). Thus, if $\gamma(s)$ has singularity at $s=0$, we can still replace $\gamma$ by $\tilde{\gamma}$ in the way as illustrated in Remark \ref{singularity0} to consider global existence on any time interval $[0,T]$. But we cannot discuss uniform-in-time boundedness in this case.
\end{remark}
Next, we recall the following lemma given in \cite[Lemma 2.4]{FS2016}.
\begin{lemma}\label{fs}Let $n=2$ and $p\in(1,2)$. There exists $K_{Sob}>0$ such that
for all $s>1$ and for all $t\in [0,T_{\mathrm{max}})$,
\begin{align*}
\io u^{p+1}
\leq
\dfrac{{K_{Sob}(p+1)}^2}{\log s}\io (u\log u +e^{-1})
\io u^{p-2} {|\nabla u|}^2
+6s^{p+1} |\Omega| +4K_{Sob}^2{|\Omega|}^{2-p}\|u_0\|^{p+1}_{L^1(\Omega)}.
\end{align*}
\end{lemma}
In addition, we need  the following uniform Gronwall inequality \cite[Chapter III, Lemma 1.1]{Temam} to deduce uniform-in-time estimates for the solutions.
\begin{lemma}\label{uniformGronwall}
	Let $g,h,y$ be three positive locally integrable functions on $(t_0,\infty)$ such that $y'$ is locally integrable on $(t_0,\infty)$ and the following inequalities are satisfied:
	\begin{equation*}
	y'(t)\leq g(t)y(t)+h(t)\;\;\forall\;t\geq t_0,
	\end{equation*}
	\begin{equation*}
	\int_t^{t+r}g(s)ds\leq a_1,\;\;\int_t^{t+r}h(s)ds\leq a_2,\;\;\int_t^{t+r}y(s)ds\leq a_3,\;\;\forall \;t\geq t_0
	\end{equation*}	where $r,a_i$, $(i=1,2,3)$ are positive constants. Then
	\begin{equation*}
	y(t+r)\leq \left(\frac{a_3}{r}+a_2\right)e^{a_1},\;\;\forall t\geq t_0.
	\end{equation*}
\end{lemma}

\section{Point-wise Upper Bounds for $v$}
In this section, we derive point-wise upper bounds for $v$ which is the key step of our studies.
\begin{lemma}\label{keylem1}Assume $n\geq1$. For any $t<T_{\mathrm{max}}$, there holds
	\begin{equation}\label{var0}
	v_t+\gamma(v)u+\mu(I-\Delta)^{-1}[u^2]=(I-\Delta)^{-1}[\gamma(v)u+\mu u].
	\end{equation}
Moreover, for  any $x\in\Omega$ and  $t\in[0,T_{\mathrm{max}})$, we have
\begin{equation}\label{keyA}
v(t,x)\leq v_0(x)+\int_0^t(I-\Delta)^{-1}[\gamma(v)u+\mu u]ds,
\end{equation}where $v_0\triangleq(I-\Delta)^{-1}u_0$.
\end{lemma}
\begin{proof}First, a substitution of the second equation into the first one yields that
	\begin{equation}\label{var}
		-\Delta v_t+v_t=\Delta (\gamma(v)u)+\mu u(1-u).
	\end{equation}
Then taking $(I-\Delta)^{-1}$ on both sides of the above equality, we obtain the following key identity:
\begin{equation}\label{var0a}
	v_t+\gamma(v)u+\mu(I-\Delta)^{-1}[u^2]=(I-\Delta)^{-1}[\gamma(v)u+\mu u].
\end{equation} 
In addition, we observe that $\gamma(v)u\geq0$ and due to the maximum principle, $(I-\Delta )^{-1}[u^2]$ is non-negative as well. 
Then \eqref{keyA} follows from a direct integration with respect to time.\end{proof}
Thanks to the preceding lemma, one easily deduce the following point-wise upper bound for $v$.
\begin{lemma}\label{pte1}Assume $n\geq1.$ For any $x\in\Omega$ and $0\leq t<T_{\mathrm{max}}$, there  holds
	\begin{equation}\label{ptest}
	v(x,t)\leq v_0(x)e^{(\gamma(v_*)+\mu)t}.
	\end{equation}
\end{lemma}
\begin{proof} Note that $v$ is non-negative due to the maximum principles.  Since $\gamma$ is non-increasing in $v$, there holds $\gamma(v)\leq \gamma(v_*)$ for all $(x,t)\in\Omega\times[0,T_\mathrm{max})$. Here and in the sequel, we assume $v_*=0$ if $\mu>0$  and $v_*>0$ if $\mu=0$ due to Lemma \ref{lowerbound}. Thus, there holds $0<\gamma(v)u\leq \gamma(v_*)u$ for any $(x,t)\in\Omega\times[0,T_{\mathrm{max}})$. Then applying the comparison principle, we deduce from the second equation of  \eqref{chemo1} that \begin{equation*}
		0\leq(I-\Delta)^{-1}[\gamma(v)u]\leq (I-\Delta)^{-1}[\gamma(v_*)u]=\gamma(v_*)v.
	\end{equation*}
As a result, we obtain from \eqref{keyA} that for any $(x,t)\in\Omega\times[0,T_{\mathrm{max}})$,
\begin{align}\nn
	v(x,t)\leq v_0(x)+\left (\gamma(v_*)+\mu\right)\int_0^tv(s,x) ds
\end{align}	
which entails \eqref{ptest} by Gronwall's inequality. This completes the proof.
\end{proof}
Therefore, $v(t,x)$ grows at an exponential rate in time at most for any $n\geq1.$ For $n\leq3$, we can improve the above estimates thanks to the Sobolev embeddings. To this aim, we need to derive some estimates for $u\gamma(v)$. 
\begin{lemma}\label{est0}Assume $n\geq1$ and $\mu=0$. There exist $C>0$ depending on the $\|u_0\|_{L^1(\Omega)}$ and $\Omega$ such that for any $t\in[0,T_{\mathrm{max}})$,
	\begin{equation}\label{L2u}
	\|u(t)-\overline{u_0}\|_{H^{-1}}^2+\|v(t)\|_{H^1(\Omega)}^2+\int_0^t\int_\Omega \gamma(v)u^2dxds\leq 2\|u_0-\overline{u_0}\|^2_{H^{-1}(\Omega)}+2\overline{u_0}^2|\Omega| +Ct,
	\end{equation}
	where $\overline{\varphi}\triangleq\frac{1}{|\Omega|}\int_\Omega \varphi dx$ for any $\varphi\in L^1(\Omega)$.	
\end{lemma}
\begin{proof} First, one verifies that if $\mu=0$, $\overline{u}(t)=\overline{v}(t)=\overline{u_0}$.
	Multiplying the first equation by $(-\Delta)^{-1}(u-\overline{u_0})$ and integrating over $\Omega$, we obtain that
	\begin{equation}
	\frac{1}{2}\frac{d}{dt}\|(-\Delta)^{-\frac12}(u-\overline{u_0})\|_{L^2(\Omega)}^2+\int_\Omega \gamma(v)u^2dx=\overline{u_0}\int_\Omega \gamma(v)udx.\nn
	\end{equation}	
Thanks to the fact that $\gamma(v)\leq\gamma (v_*)$, we conclude that
	\begin{equation*}
	\frac{1}{2}\frac{d}{dt}\|(-\Delta)^{-\frac12}(u-\overline{u_0})\|_{L^2(\Omega)}^2+\int_\Omega \gamma(v)u^2dx\leq \gamma(v_*)\overline{u_0}^2|\Omega|,
	\end{equation*}which  by a direct integration with respect to time implies that for any $t\in(0,T_{\mathrm{max}})$
	\begin{equation*}
	\|(-\Delta)^{-\frac12}(u(t)-\overline{u_0})\|_{L^2(\Omega)}^2+2\int_0^t\int_\Omega \gamma(v)u^2dx\leq \|(-\Delta)^{-\frac12}(u_0-\overline{u_0})\|_{L^2(\Omega)}^2+2 \gamma(v_*)\overline{u_0}^2|\Omega|t.
	\end{equation*}
	On the other hand, we observe from the second equation of \eqref{chemo1} that
	\begin{align*}
	\|v\|_{H^1(\Omega)}^2=&\int_\Omega (|\nabla v|^2+v^2)dx\non\\
	=&\int_\Omega uvdx\non\\
	=&\int_\Omega (u-\overline{u_0})vdx+\overline{u_0}^2|\Omega|\non\\
	\leq& \|u-\overline{u_0}\|_{H^{-1}(\Omega)}\|v\|_{H^1(\Omega)}+\overline{u_0}^2|\Omega|.
	\end{align*} 
	Thus, by Young's inequality, we obtain that
	\begin{equation*}
	\|v\|_{H^1(\Omega)}^2\leq \|u-\overline{u_0}\|^2_{H^{-1}(\Omega)}+2\overline{u_0}^2|\Omega|,
	\end{equation*}	which completes the proof.
\end{proof}
On the other hand, when $\mu>0$, one can derive the following estimates.
\begin{lemma} \label{mu1}Assume $n\geq1$ and $\mu>0$.
	Let $(u,v)$ be a classical solution of system \eqref{chemo1} on $\Omega\times (0,T_{\mathrm{max}}).$ Then there is $C>0$ depending only on $\|u_0\|_{L^1(\Omega)}$ and $\Omega$ such that
	\begin{equation}\label{unib5}
	\sup\limits_{0<t<T_{\mathrm{max}}}\int_\Omega u dx\leq \max\left\{\int_\Omega u_0,|\Omega|\right\},
	\end{equation}and for any $t\in(0,T_{\mathrm{max}}-\tau)$ with any fixed $0<\tau<\min\{1,T_{\mathrm{max}}/2\}$,
	\begin{equation}\label{unib4}
		\int_t^{t+\tau}\int_\Omega u^2dxds\leq C+C/\mu.
	\end{equation}
Moreover, we have
	\begin{equation}\label{unil2b}
	\sup\limits_{0<t<T_{\mathrm{max}}}\int_0^t e^{\mu(s-t)}\|u(s)\|_{L^2(\Omega)}^2ds\leq C/\mu.
	\end{equation}
\end{lemma}
\begin{proof} The former two assertions were given in \cite[Lemma 2.2]{JKW18}. Integrating over $\Omega$, adding $\mu\int_\Omega udx$ to both sides of the first equation in \eqref{chemo1}, and applying Young's inequality, we obtain that
	\begin{equation*}
	\frac{d}{dt}\left[e^{\mu t}\int_\Omega udx\right]+\mu e^{\mu t}\int_\Omega u^2dx=2\mu e^{\mu t}\int_\Omega udx\leq \frac{\mu}{2}e^{\mu t}\int_\Omega u^2dx+2\mu|\Omega|e^{\mu t}.
	\end{equation*}
	Then an integration of the above inequality with respect to time entails that
	\begin{align*}
	e^{\mu t}\int_\Omega u(t)dx+\frac{\mu}{2}\int_0^t e^{\mu s}\|u(s)\|_{L^2(\Omega)}^2ds	\leq \int_\Omega u_0dx+ 2|\Omega|(e^{\mu t}-1)
	\end{align*}	
	which yields \eqref{unil2b} by dividing the above inequality by $e^{\mu t}$. This completes the proof.	
\end{proof}
With the above two lemmas at hand, we can improve the point-wise upper bound for $v$ when $n\leq3.$ First, if $\mu=0$, we prove that the growth rate of $v$ is at most linear in time.
\begin{lemma}
	Assume $n\leq3$ and $\mu=0$. There exists $C>0$ depending only on $\|u_0\|_{L^1(\Omega)}$ and $\Omega$ such that for any $(x,t)\in\Omega\times[0,T_{\mathrm{max}})$
	\begin{equation*}
	v(x,t)\leq v_0(x)+C\|u_0-\overline{u_0}\|^2_{H^{-1}(\Omega)}+C(t+1).
	\end{equation*}
\end{lemma}
\begin{proof}
	Thanks to Lemma \ref{est0} and the three-dimensional Sobolev embedding theorem, we infer that
	\begin{align*}
		\int_0^t (I-\Delta)^{-1}[\gamma(v)u]ds\leq&\int_0^t\|(I-\Delta)^{-1}[\gamma(v)u]\|_{L^\infty(\Omega)}ds\non\\
		\leq&C\int_0^t\|\gamma(v)u\|_{L^2(\Omega)}ds\non\\
		\leq&C\int_0^t(\|\gamma(v)u\|_{L^2(\Omega)}^2+1)ds\non\\
		\leq&C\gamma(v_*)\int_0^t\int_\Omega \gamma(v)u^2dxds+Ct\non\\
		\leq&C\|u_0-\overline{u_0}\|^2_{H^{-1}(\Omega)}+C(t+1).
	\end{align*}
	Therefore, invoking \eqref{keyA}, we deduce that
	\begin{equation}
		v(x,t)\leq v_0(x)+C\|u_0-\overline{u_0}\|^2_{H^{-1}(\Omega)}+C(t+1).\non
	\end{equation}This completes the proof.
\end{proof}
In contrast, if $\mu>0$, invoking Lemma \ref{mu1}, we may prove the following uniform-in-time upper bound for $v$.
\begin{lemma}\label{pte2}
	Assume $n\leq3$ and $\mu>0$. There exists $C>0$ depending only on $\|u_0\|_{L^1(\Omega)}$ and $\Omega$ such that for any $(x,t)\in\Omega\times[0,T_{\mathrm{max}})$
	\begin{equation*}
		v(x,t)\leq v_0(x)+1+\frac{C}{\mu}.
	\end{equation*}
\end{lemma}
\begin{proof}
	Since $2u\leq u^2+1$ and $v=(I-\Delta)^{-1}u,$ we deduce by the comparison principle that
	\begin{equation}\label{com}
	v=(I-\Delta)^{-1}[u]\leq (I-\Delta)^{-1}[u^2-u+1]=(I-\Delta)^{-1}[u^2-u]+1.
	\end{equation}
It follows from \eqref{var0} that
\begin{equation}\label{keymu}
	v_t+\mu v+\gamma(v)u\leq (I-\Delta)^{-1}[\gamma(v)u]+\mu
\end{equation}
which entails that
\begin{equation*}
	v(x,t)\leq v_0(x)+1+e^{-\mu t}\int_0^t(I-\Delta)^{-1}[\gamma(v)u] e^{\mu s}ds.
\end{equation*}
Finally, we observe that due to the three-dimensional Sobolev embedding theorem and \eqref{unil2b}, there holds
\begin{align*}
	e^{-\mu t}\int_0^t(I-\Delta)^{-1}[\gamma(v)u] e^{\mu s}ds\leq& e^{-\mu t}\int_0^te^{\mu s}\|(I-\Delta)^{-1}[\gamma(v)u]\|_{L^\infty(\Omega)}ds\non\\
	\leq& Ce^{-\mu t}\int_0^te^{\mu s}\|\gamma(v)u\|_{L^2(\Omega)}ds\non\\
	\leq&C\gamma(v_*)e^{-\mu t}\int_0^te^{\mu s}(\|u\|^2_{L^2(\Omega)}+1)ds\non\\
	\leq&\frac{C}{\mu}.
\end{align*}This completes the proof.
\end{proof}
Furthermore, if $\mu$ is enough large, then uniform-in-time upper bound for $v$ is available in any dimensions.
\begin{lemma}\label{pte3}
	Assume $n\ge1$. Then if $\mu>\gamma(v_*)$, there holds for any $(x,t)\in\Omega\times[0,T_{\mathrm{max}})$ that
	\begin{equation}\label{muest2}
		v(x,t)\leq v_0(x)+\frac{\mu}{\mu-\gamma(v_*)}.
	\end{equation}
\end{lemma}
\begin{proof}
Since $\gamma(v)\leq \gamma(v_*)$, we infer by the comparison principle that
\begin{equation*}
	0\leq(I-\Delta)^{-1}[\gamma(v)u]\leq \gamma(v_*)(I-\Delta)^{-1}[u]=\gamma(v_*)v.
\end{equation*}	As a result, we infer from \eqref{keymu} that
\begin{equation*}
	v_t+(\mu-\gamma(v_*)) v\leq \mu
\end{equation*}and thus for any fixed $x\in\Omega$
\begin{equation*}
	\frac{d}{d t}[e^{t(\mu-\gamma(v_*))}v(x,t)]\leq \mu e^{t(\mu-\gamma(v_*))}
\end{equation*}
which yields \eqref{muest2} by a direct integration with respect to time.
\end{proof}
\section{Global Existence with General Motilities}
In this section, we study system \eqref{chemo1} with general motility functions. First, thanks to Lemma \ref{pte1}, we establish global existence of classical solutions with general motility functions satisfying \eqref{gamma} and $\mu=0$ in the two-dimensional setting. Then, in view of uniform-in-time upper boundedness of $v$ given by Lemma \ref{pte2}, we study existence of classical solutions with uniform-in-time bounds when $\mu>0.$ Last, we prove uniform-in-time boundedness when $\gamma$ satisfies the extra growth condition (A2). 
\subsection{Global Existence when $\mu=0$ in 2D}
Since we have upper bound of $v$, we can argue now in a similar way as done for classical Keller-Segel models. First, we have
\begin{lemma}\label{glm1}
	 Assume $(u,v)$ is a classical solution of system \eqref{chemo1} on $\Omega\times (0,T)$. Then there exists $C(T)>0$ depending on the $u_0,\Omega$ and $T$ such that 
	\begin{equation*}
		\sup\limits_{0<t<T}\int_\Omega u(t)\log u(t) dx+\int_0^T\int_\Omega(1+\gamma(v))\frac{|\nabla u|^2}{u}dxds\leq C(T).
	\end{equation*}
\end{lemma}
\begin{proof}
	Multiplying the first equation of \eqref{chemo1} by $\log u$ and integrating over $\Omega$, we obtain that
\begin{align}
	\frac{d}{dt}\int_\Omega u\log udx+\int_\Omega \gamma(v)\frac{|\nabla u|^2}{u}dx=&-\int_\Omega \gamma'(v)\nabla v\cdot \nabla udx\non\\
	\leq&\frac12\int_\Omega \gamma(v)\frac{|\nabla u|^2}{u}dx+\int_\Omega 
	\frac{|\gamma'(v)|^2}{\gamma(v)}u|\nabla v|^2dx\non\\
	\leq &\frac12\int_\Omega \gamma(v)\frac{|\nabla u|^2}{u}dx+\int_\Omega \gamma(v)u^2dx+\int_\Omega 
	\frac{|\gamma'(v)|^4}{\gamma(v)^3}|\nabla v|^4dx.\non
\end{align}	
	Notice that for $n\leq3$, the Sobolev embedding indicates that
	\begin{equation}\nn
		\|\nabla v\|_{L^4(\Omega)}\leq C\|v\|_{H^2(\Omega)}^{1/2}\|v\|^{1/2}_{L^\infty(\Omega)}+C\|v\|_{L^\infty(\Omega)}.
	\end{equation}Therefore, in view of Lemma \ref{lowerbound}, Lemma \ref{pte1} and our assumption \eqref{gamma} on $\gamma,$ there is $C(T)$ depending on $v_*$ and $\gamma$
 such that\begin{align}
		\int_\Omega\frac{|\gamma'(v)|^4}{\gamma(v)^3}|\nabla v|^4dx\leq& C(T)\int_\Omega|\nabla v|^4dx\non\\
		\leq&C(T)\|v\|^2_{H^2(\Omega)}+C(T).\nn
\end{align}	On the other hand, since $\gamma(v)$ is now bounded from below, we observe from the elliptic regularity theorem and Lemma \ref{est0} that
\begin{equation}
	\int_0^T\|v\|^2_{H^2(\Omega)}dt\leq C\int_0^T\|u\|_{L^2(\Omega)}^2dt\leq C(T).\nn
\end{equation}
	Finally, we deduce that
	\begin{equation}
	\int_\Omega u\log udx+\int_0^T\int_\Omega (1+\gamma(v))\frac{|\nabla u|^2}{u}dxdt	\leq C(T)\nn
	\end{equation}
	which completes the proof.
\end{proof}

\begin{lemma}\label{lemma_LP}
	Assume that  $(u,v)$ is a classical solution of system \eqref{chemo1} on $\Omega \times (0,T)$. 
	Then there exist  $p \in (1,2)$ and some $C(T) >0$ such that
	\begin{eqnarray*}
	\|u(t)\|_{L^p(\Omega)} \leq C(T)\qquad
	\mbox{for all }t\in(0,T).
	\end{eqnarray*}
\end{lemma}
\begin{proof}
	Multiplying the first equation of \eqref{chemo1} by $ u^{p-1}$ we have
	\begin{eqnarray*}
		\frac{1}{p}\frac{d}{dt} \io u^p \,dx 
		&=& \io u^{p-1} u_t \,dx\\
		&=& \io u^{p-1} \nabla \cdot (\gamma(v) \nabla u + u \gamma'(v) \nabla v)\,dx,
	\end{eqnarray*}
	and by integration by parts, it follows that
	\begin{eqnarray*}
		\frac{1}{p}\frac{d}{dt} \io u^p \,dx +(p-1) \io u^{p-2}\gamma(v)|\nabla u|^2 \,dx
		=
		-(p-1) \io u^{p-1}\gamma'(v) \nabla u \cdot \nabla v \,dx.
	\end{eqnarray*}
	By the Cauchy-Schwarz inequality we have
	\begin{eqnarray*}
		\frac{1}{p}\frac{d}{dt} \io u^p \,dx +\frac{p-1}{2} \io u^{p-2}\gamma(v)|\nabla u|^2 \,dx
		&\leq&
		\frac{p-1}{2} \io \frac{u^{p}|\gamma'(v)|^2}{\gamma(v)} |\nabla v|^2 \,dx \\
		&\leq &
		pM_{\gamma}(T) \io u^{p}|\nabla v|^2 \,dx,
	\end{eqnarray*}
	where we set
\begin{equation*}
	M_{\gamma}(T)=\sup\limits_{s\in[v_*, v^*(T)]}\frac{|\gamma'(s)|^2}{\gamma(s)} 
\end{equation*}
with $v^*(T)\triangleq e^{\gamma(v_*)T}\|v_0\|_{L^\infty}$.
	Using H\"older's inequality and Young's inequality we obtain that
	\begin{align*}
	\io u^p {|\nabla v|}^2 \,dx
	&\leq
	{\bigg(\io u^{p+1}\,dx \bigg)}^{\frac{p}{p+1}}
	{\bigg(\io {|\nabla v|}^{2(p+1)}\,dx \bigg)}^{\frac{1}{p+1}}\\
	&\leq
	\dfrac{p}{p+1}\io u^{p+1} v \,dx
	+
	\dfrac{1}{p+1}\io {|\nabla v|}^{2(p+1)} \,dx,
	\end{align*}
	and in view of Lemma \ref{ptest}, we obtain
	\begin{align*}
	\dfrac{1}{p}\dfrac{d}{dt}\io u^p \,dx 
	+C\io u^{p-2}{|\nabla u|}^2 \,dx
	\leq
	C \io u^{p+1} \,dx
	+ C \io {|\nabla v|}^{2(p+1)}\,dx,
	\end{align*}
	with some $C=C(T)>0$.

On the other hand, by the Sobolev embedding theorem and the elliptic regularity theory, we deduce that
	\begin{align*}
	\|\nabla v\|_{L^{2(p+1)}(\Omega)}
	\leq
	C \|v\|_{W^{2,\frac{2(p+1)}{p+2}}(\Omega)}
	\leq
	C \|(-\Delta+1) v\|_{L^{\frac{2(p+1)}{p+2}}(\Omega)}=C\|u\|_{L^{\frac{2(p+1)}{p+2}}(\Omega)}
	\end{align*}
	with positive constants $C$. By the interpolation inequality, there holds 
	\begin{align*}
	{\|u\|}^{2(p+1)}_{L^{\frac{2(p+1)}{p+2}}(\Omega)}
	\leq
	\|u\|^{p+1}_{L^1(\Omega)}\io u^{p+1}.
	\end{align*}
	Therefore we have 
	\begin{align}\label{Lpineq}
		\dfrac{d}{dt}\io u^p dx
		+C\io u^{p-2}{|\nabla u|}^2 dx
		\leq
		C \io u^{p+1} dx.
	\end{align}
	Finally picking $s>0$ sufficiently large in Lemma \ref{fs} and recalling Lemma \ref{glm1}, we obtain that
	\begin{align*}
			\dfrac{d}{dt}\io u^p \,dx 
		+C\io u^{p-2}{|\nabla u|}^2 \,dx
		\leq
		C (T).
	\end{align*}
	Thus we complete the proof by a direct integration with respect to time.
\end{proof}
After the above preparation, we may use standard bootstrap argument to prove that 
\begin{equation*}
	\sup\limits_{0<t<T}\|u(\cdot, t)\|_{L^\infty(\Omega)}\leq C(T)
\end{equation*}for any $T<T_{\mathrm{max}}$
and hence by Theorem \ref{local}, we deduce that $T_{\mathrm{max}}=+\infty$. Therefore, we have
\begin{proposition}\label{Prop4.1}
	Let $\Omega\subset\mathbb{R}^2$ and $\mu=0$. Assume \eqref{ini} and \eqref{gamma} are satisfied. Then there exists a unique global classical solution $(u,v)$ to system \eqref{chemo1}.
\end{proposition}

\subsection{Uniform-in-time Boundedness when $\mu>0$ in 2D}
In this part, we consider the case $\mu>0$. First, we have
\begin{lemma}
	Assume $n\leq3$ and  $\mu>0$. Suppose $(u,v)$ is a classical solution on $\Omega\times(0,T_{\mathrm{max}})$. There exists $C>0$ which is independent of $T_{\mathrm{max}}$ such that
	\begin{equation}\label{unib1}
	\sup\limits_{0<t<T_{\mathrm{max}}}\int_\Omega (u\log u+|\nabla v|^2+v^2) dx\leq C
	\end{equation}and for any $t\in(0,T_{\mathrm{max}}-\tau)$ with any fixed $0<\tau<\min\{1,T_{\mathrm{max}}/2\}$,
	\begin{equation}\label{unib2}
		\int_t^{t+\tau}\int_\Omega(|\Delta v|^2+|\nabla v|^4) dxds\leq C+C/\mu.
	\end{equation}
\end{lemma}
\begin{proof}
	Using the second equation of \eqref{chemo1} and \eqref{var0}, we notice that
	\begin{align}
	\frac12\frac{d}{dt}\int_\Omega(|\nabla v|^2+v^2)dx=&\int_\Omega uv_t dx\non\\
	=&\int_\Omega u(I-\Delta)^{-1}[u\gamma(v)]-\int_\Omega u^2\gamma(v)dx-\mu\int_\Omega u(I-\Delta)^{-1}[u^2-u]dx.\non
	\end{align}
Then, it follows from \eqref{com} that
	\begin{equation}\non
	\frac12\frac{d}{dt}\int_\Omega(|\nabla v|^2+v^2)dx+\int_\Omega u^2\gamma(v)dx+\mu\int_\Omega uv\leq\int_\Omega u(I-\Delta)^{-1}[u\gamma(v)]dx+\mu\int_\Omega udx.
	\end{equation}	
	If $n\leq3, $ we apply the Sobolev embedding theorem to deduce that
	\begin{align*}\non
	\int_\Omega u(I-\Delta)^{-1}[u\gamma(v)]dx\leq&\|(I-\Delta)^{-1}[u\gamma(v)]\|_{L^\infty(\Omega)}\int_\Omega udx\non\\
	\leq&	C\|u\gamma(v)\|_{L^2(\Omega)}\int_\Omega udx.
	\end{align*}
	In view of the Cauchy-Schwarz inequality, for any $\varepsilon>0$ it follows that
		\begin{align}\non
	\int_\Omega u(I-\Delta)^{-1}[u\gamma(v)]dx
	\leq\varepsilon\gamma(v_*)\int_\Omega u^2\gamma(v)dx+\frac{C}{\varepsilon}\left(\int_\Omega udx\right)^2.\non
	\end{align}
Since the second equation of \eqref{chemo1} implies
\begin{equation}\non
	\int_\Omega uv dx=\int_\Omega (|\nabla v|^2+ v^2)dx,
\end{equation}	
by choosing small $\varepsilon>0$ we deduce from above inequalities that	
\begin{align}\label{unib3}
\non
	\frac{d}{dt}\int_\Omega(|\nabla v|^2+v^2)dx+2\mu\int_\Omega (|\nabla v|^2+ v^2)dx
	\leq& C\left(\int_\Omega udx\right)^2 +2\mu\int_\Omega udx\\
	\leq& C(\Lambda^2+\Lambda).
\end{align}
Hence,  by solving \eqref{unib3} we obtain that
\begin{equation}\label{unib1a}
\sup\limits_{0<t<T_{\mathrm{max}}}\int_\Omega (|\nabla v|^2+v^2) dx\leq C.
\end{equation} On the other hand, the elliptic regularity theorem together with \eqref{unib4} yields that
\begin{equation}\nonumber
\int_t^{t+\tau}\|v\|^2_{H^2(\Omega)}ds\leq C+C/\mu.
\end{equation}
Thus \eqref{unib2} follows from Lemma \ref{pte2}, \eqref{unib1a} and the Sobolev embedding
\begin{equation*}
\|\nabla v\|^4_{L^4(\Omega)}\leq C\|v\|_{H^2(\Omega)}^2\|v\|_{L^\infty(\Omega)}^2+C\|\nabla v\|_{L^2(\Omega)}^4.
\end{equation*}
Next, multiplying the first equation of \eqref{chemo1} by $\log u$, we see that
\begin{align*}
	&\frac{d}{dt}\left(\int_\Omega u\log u dx-\int_\Omega udx\right)+\int_\Omega \frac{\gamma(v)|\nabla u|^2}{u}dx\\
	=&-\int_\Omega \gamma'(v)\nabla v\cdot \nabla u+\mu\int_\Omega (u\log u-u^2\log u)dx.
\end{align*}
Since the integration by parts implies
\begin{align*}
-\int_\Omega \gamma'(v)\nabla v\cdot \nabla u
=\int_\Omega \gamma''(v)u|\nabla v|^2dx+\int_\Omega \gamma'(v)u\Delta vdx,
\end{align*}
then by adding $\int_\Omega u^2 dx$ both sides and using the Young's inequality, we have
\begin{align*}
	\frac{d}{dt}&\left(\int_\Omega u\log u dx-\int_\Omega udx\right)+\int_\Omega \frac{\gamma(v)|\nabla u|^2}{u}dx+\int_\Omega u^2 dx\non\\
	=&\int_\Omega u^2 dx+\int_\Omega \gamma''(v)u|\nabla v|^2dx+\int_\Omega \gamma'(v)u\Delta vdx	+\mu\int_\Omega (u\log u-u^2\log u)dx\non\\
	\leq&C_\gamma(\mu)\int_\Omega u^2dx+C\int_\Omega (|\Delta v|^2+|\nabla v|^4) dx+\mu\int_\Omega u\log u+\mu |\Omega|,
\end{align*}where $C_\gamma(\mu)$ depends on $K_1(\mu,\gamma)\triangleq\max\limits_{v_*\leq s\leq v^*}\{\gamma''(s),\gamma'(s)\}$ with $v^*$ denoting the uniform-in-time $L^\infty$-bound obtained in Lemma \ref{pte2}. As $\xi\log\xi-\xi\leq \xi^2$ for all $\xi>0$ and hence
\begin{equation}
	\non \int_\Omega (u\log u-u)dx\leq \int_\Omega u^2dx,
\end{equation} thanks to \eqref{unib2}, \eqref{unib5} and \eqref{unib4}, we infer by means of ODE analysis that
\begin{equation}\non
	\int_\Omega u\log udx\leq C.
\end{equation}This completes the proof.
\end{proof}
Once we establish \eqref{unib1}, we can argue as before to obtain uniform-in-time boundedness of $\|u\|_{L^\infty(\Omega)}$ in the 2D case \cite{BBTW15}. In addition, one can argue in the same way as in \cite{JKW18} to obtain the stability of the classical solutions. Thus, we have
\begin{proposition}\label{Prop4.2}
	Assume $n=2$ and $\mu>0$. For any $u_0$ and $\gamma(\cdot)$ satisfying \eqref{ini} and \eqref{gamma} respectively, system \eqref{chemo1} has a unique classical solution $(u,v)$ that is uniform-in-time bounded.
	
	 Moreover, if
	\begin{equation*}
		K_0\triangleq\max\limits_{0\leq s\leq +\infty}\frac{|\gamma'(s)|^2}{\gamma(s)}<+\infty
	\end{equation*}
and $\mu>\frac{K_0}{16}$,	there holds
	\begin{equation*}
		\lim\limits_{t\rightarrow+\infty}\left(\|u(\cdot,t)-1\|_{L^\infty(\Omega)}+\|v(\cdot,t)-1\|_{L^\infty(\Omega)}\right)=0.
	\end{equation*}
\end{proposition}

\begin{proof}[Proof of Theorem \ref{TH1}]
The proof concludes from Proposition \ref{Prop4.1} and Proposition \ref{Prop4.2}.
\end{proof}

\subsection{Proof of Theorem \ref{TH3}}
In this part, we prove Theorem \ref{TH3}.  To this aim, it is suffices to show the following uniform boundedness of $v(x,t)$ since the rest part of proof is the same as in \cite{TaoWin17}.
\begin{proposition}\label{Propvest1}
	Under the assumption of Theorem \ref{TH3}, there exists $C>0$ depending only on $\|u_0\|_{L^1(\Omega)}$, $\gamma$ and $\Omega$ such that for all $t\geq0$ and $x\in\Omega,$
	\begin{equation*}
		v(x,t)\leq C.
	\end{equation*}
\end{proposition}
\begin{proof}
	First, multiplying the first equation  by $v$ and making use of the second equation of \eqref{chemo1}, we infer that
	\begin{align*}
	\frac{1}{2}\frac{d}{dt}\int_\Omega\left(|\nabla v|^2+v^2\right)dx&=\int_\Omega u\gamma(v)\Delta v dx\\
	&=\int_\Omega u\gamma(v)(v-u)dx\\
	&\leq-\frac12\int_\Omega u^2\gamma(v)dx+\frac12\int_\Omega v^2\gamma(v)dx.
	\end{align*}
Therefore, we obtain that
\begin{equation}\label{unibb1}
	\frac{d}{dt}\int_\Omega \left(|\nabla v|^2+v^2\right)dx+\int_\Omega u^2\gamma(v)dx\leq\int_\Omega v^2\gamma(v)dx.
\end{equation}
Next, thanks to the second equation of \eqref{chemo1} again, we observe that
\begin{align*}
	\|\nabla v\|_{L^2(\Omega)}^2=&-\int_\Omega v\Delta vdx\\
	\leq&\frac12\int_\Omega|\Delta v|^2\gamma(v)dx+\frac12\int_\Omega \frac{v^2}{\gamma(v)}dx\\
	\leq&\int_\Omega u^2\gamma(v)dx+\int_\Omega v^2\gamma(v)dx+\frac12\int_\Omega \frac{v^2}{\gamma(v)}dx.
\end{align*}
It follows from above inequalities that
\begin{equation}\label{unibb2}
	\frac{d}{dt}\int_\Omega \left(|\nabla v|^2+v^2\right)dx+\int_\Omega  \left(|\nabla v|^2+v^2\right)dx\leq2\int_\Omega v^2\gamma(v)dx+\frac12\int_\Omega \frac{v^2}{\gamma(v)}dx+\int_\Omega v^2dx.
\end{equation}
Now, due to our assumption (A2), we may infer that there exist $k>0$, $b>0$ and $s_b>v_*$ such that for all $s\geq s_b$
\begin{equation*}
	\gamma^{-1}(s)\leq bs^k
\end{equation*}and on the other hand, since $\gamma(\cdot)$ is decreasing,
\begin{equation*}
	\gamma^{-1}(s)\leq \gamma^{-1}(s_b)
\end{equation*}for all $0<s<s_b$.
Therefore, for all $s>0$, there holds
\begin{equation}\label{cond_gamma}
\gamma^{-1}(s)\leq bs^{k}+\gamma^{-1}(s_b).
\end{equation}
Then recalling the elliptic regularity estimate \cite{BS}
\begin{equation}\label{unibb3}
	\|v\|_{L^p(\Omega)}\leq C\|u\|_{L^1(\Omega)}=C\|u_0\|_{L^1(\Omega)}
\end{equation}for any $1\leq p<\infty$ when $n=2$, one can find $C>0$ depending on $\|u_0\|_{L^1(\Omega)}$, $\gamma$ and $\Omega$ such that
\begin{equation*}
	2\int_\Omega v^2\gamma(v)dx+\frac12\int_\Omega \frac{v^2}{\gamma(v)}dx+\int_\Omega v^2dx\leq C.
\end{equation*}
Then solving the differential inequality \eqref{unibb2} yields that
\begin{equation*}
	\sup\limits_{t\geq0}\left(\|\nabla v\|_{L^2(\Omega)}^2+\|v\|_{L^2(\Omega)}^2\right)\leq C.
\end{equation*}
In addition, a direct integration of \eqref{unibb1} with respect to time from $t$ to $t+1$ indicates that
\begin{equation}\label{estts_u}
	\sup\limits_{t>0}\int_t^{t+1}\int_\Omega u^2\gamma(v)dxds\leq C
\end{equation}with $C>0$ depending on $\|u_0\|_{L^1(\Omega)}$, $\gamma$ and $\Omega$ only.

Now, for any $1<p<2$, we infer by the Sobolev embedding theorem that
\begin{align*}
	\|v\|_{L^\infty(\Omega)}\leq &C\|u\|_{L^p(\Omega)}\\
	\leq&\left(\int_\Omega u^2\gamma(v)dx\right)^{\frac12}\left(\int_\Omega \gamma^{-\frac{p}{2-p}}(v)dx\right)^{\frac{2-p}{2p}}\\
	\leq&C\left(\int_\Omega u^2\gamma(v)dx\right)^{\frac12},
\end{align*}where we use \eqref{cond_gamma} and \eqref{unibb3} to deduce that
\begin{equation*}
	\int_\Omega \gamma^{-\frac{p}{2-p}}(v)dx\leq \int_\Omega \left(bv^k+\gamma^{-1}(s_b)\right)^{\frac{p}{2-p}}dx\leq C.
\end{equation*}
Then by \eqref{estts_u}, for any $t>0$ we obtain that
\begin{equation*}
	\int_t^{t+1}\|v\|_{L^\infty(\Omega)}ds\leq \int_t^{t+1}\int_\Omega u^2\gamma(v)dx+1\leq C
\end{equation*}
and thus for any fixed $x\in\Omega$,
\begin{equation*}
	\sup\limits_{t>0}\int_t^{t+1}v(s,x)ds\leq C
\end{equation*}
with $C>0$ depending only on $\|u_0\|_{L^1(\Omega)}$, $\gamma$ and $\Omega$. Finally, 
observing that
\begin{equation}\non
v_t+u\gamma(v)=(I-\Delta)^{-1}[u\gamma(v)]\leq \gamma(v_*)(I-\Delta)^{-1}[u]=\gamma(v_*)v,
\end{equation}
we may apply the uniform Gronwall inequality Lemma \ref{uniformGronwall} to obtain that for any $x\in\Omega$
\begin{equation*}
	v(x,t)\leq C\qquad\text{for}\;\;t\geq1,
\end{equation*}with some $C>0$ independent of $x\in\Omega$, which together with Lemma \ref{pte1} for $t\leq1$ concludes the proof.
\end{proof}

\begin{proof}[Proof of Theorem \ref{TH3}]
In light of Proposition \ref{Propvest1}, we can proceed along the same lines in \cite{TaoWin17}.
\end{proof}

\section{The Case $\gamma(v)=e^{-v}$ and $\mu=0$ in 2D}
This section is devoted to the special case $\gamma(v)=e^{-v}$ and $\mu=0$. Namely, we consider the following initial Neumann boundary value problem:
\begin{equation}
\begin{cases}\label{chemo2}
u_t=\Delta (ue^{-v})&x\in\Omega,\;t>0,\\
-\Delta v+v=u&x\in\Omega,\;t>0,\\
\partial_\nu u=\partial_\nu v=0,\qquad &x\in\partial\Omega,\;t>0,\\
u(x,0)=u_0(x),\qquad & x\in\Omega,
\end{cases}
\end{equation}with $\Omega\subset\mathbb{R}^2.$
\subsection{Uniform-in-time Boundedness with Sub-critical Mass}
In this part, we first prove the following uniform-in-time boundedness of the classical solutions with subcritical mass.
\begin{proposition}\label{prop1}
Assume $n=2$ and let
\begin{equation*}
\Lambda_c=\begin{cases}
8\pi\qquad\text{if}\;\Omega=\{x\in\mathbb{R}^2;\;|x|<R\}\;\;\text{and}\;u_0 \;\text{is radial in}\;x\\
4\pi\qquad\text{otherwise.}		
\end{cases}
\end{equation*}
If $\Lambda\triangleq\int_\Omega u_0dx<\Lambda_c$, then the global classical solution $(u,v)$ to system \eqref{chemo2} is uniform-in-time bounded in the sense that
\begin{equation*}
	\sup\limits_{t\in(0,\infty)}\left(\|u(\cdot,t)\|_{L^\infty(\Omega)}+\|v(\cdot,t)\|_{L^\infty(\Omega)}\right)<\infty.
\end{equation*}
\end{proposition}
As noticed in Introduction, system \eqref{chemo2} has the Lyapunov functional.
\begin{lemma} There holds
\begin{equation}\label{Lyapunov}
\frac{d}{dt}E(u,v)(t)+	\int_\Omega ue^{-v}\left|\nabla \log u-\nabla v\right|^2dx=0,
\end{equation} where the  functional $E(\cdot,\cdot)$ is defined by
\begin{equation*}\non
E(u,v)=\int_\Omega \left(u\log u+\frac12|\nabla v|^2+\frac12 v^2-uv\right)dx.
\end{equation*}
\end{lemma}
\begin{proof}
	Multiplying the first equation by $\log u-v$, the second equation by $v_t$ and integrating by parts, then adding the resultants together, we get
\begin{equation}\non
	\frac{d}{dt}\int_\Omega \left(u\log u+\frac12|\nabla v|^2+\frac12 v^2-uv\right)dx+\int_\Omega ue^{-v}\left|\nabla \log u-\nabla v\right|^2dx=0.
\end{equation}This completes the proof.	
\end{proof}
Since the energy $E(\cdot,\cdot)$ is the same as that of the classical Keller--Segel model, we may recall \cite[Lemma 3.4]{Nagai97} stated as follows.
\begin{lemma}\label{ublm1} 
If $\Lambda<\Lambda_c$, there exists a positive constant $C$ independent of $t$ such that
\begin{equation*}
	\int_\Omega uvdx\leq C\qquad\text{and}\qquad|E(u(t),v(t))|\leq C,\;\;\forall \;t\geq0.
\end{equation*}	
\end{lemma}
\begin{lemma}\label{ublm2}If $\Lambda<\Lambda_c$, then there holds
	\begin{equation*}
	\sup\limits_{t\geq0}\int_t^{t+1}\int_\Omega  e^{-v(s)}u^2(s)dxds\leq C,
	\end{equation*}where $C>0$ depends on $\Omega$ and the initial datum only.
\end{lemma}
\begin{proof}
	Multiplying the first equation of \eqref{chemo1} by $v$ and integrating over $\Omega$, we obtain that
	\begin{equation*}
	\int_\Omega u_t vdx=\int_\Omega e^{-v}u\Delta v dx.
	\end{equation*}
	A substitution of the second equation into the above equality implies that
	\begin{align*}
	\int_\Omega (-\Delta v_t+v_t)v dx+\int_\Omega e^{-v}u^2dx=\int_\Omega e^{-v}uvdx.
	\end{align*}
	Hence, we have
	\begin{equation}\label{unb0}
	\frac12\frac{d}{dt}(\|\nabla v\|_{L^2(\Omega)}^2+\|v\|_{L^2(\Omega)}^2)+\int_\Omega e^{-v}u^2dx=\int_\Omega e^{-v}uvdx\leq C.
	\end{equation}Then the assertion follows from  an integration with respect to time from $t$ to $t+1$ due to the fact $\|v\|_{H^1}\leq C$ when $\Lambda<\Lambda_c$ by Lemma \ref{ublm1}. This completes the proof.
\end{proof}

\begin{lemma}\label{ublm3}
	If $\Lambda<\Lambda_c$, then there exists $C>0$ depending on $\Omega$ and the initial data such that for all $x\in\Omega$
	\begin{equation*}
	\sup\limits_{t\geq0}	v(x,t)\leq C.
	\end{equation*}	
\end{lemma}

\begin{proof}
First, we apply the Sobolev embedding theorem, the elliptic regularity theorem and H\"older's inequality to infer that
\begin{align*}
\|v\|_{L^\infty(\Omega)}\leq &C\|v\|_{W^{2,\frac32}(\Omega)}\non\\
\leq&C\|u\|_{L^{\frac32}(\Omega)}\non\\
=&C\left(\int_\Omega u^{\frac32}dx\right)^{\frac23}\non\\
\leq&C\left(\int_\Omega u^2 e^{-v}dx\right)^{\frac12}\left(\int_\Omega e^{3v}dx\right)^{\frac16}\non\\
\leq&C\left(\int_\Omega u^2 e^{-v}dx\right)^{1/2},
\end{align*}
where we used the 2D Trudinger-Moser inequality \cite[Theorem 2.2]{Nagai97} to deduce that
\begin{align*}\non
	\int_\Omega e^{3v}dx\leq Ce^{C(\|\nabla v\|_{L^2(\Omega)}^2+\|v\|_{L^2(\Omega)}^2)}
\end{align*}with $C>0$ depending only on $\Omega.$
Thus, by Lemma \ref{ublm1} and Lemma \ref{ublm2}, for any $t\geq0$, there holds
\begin{equation*}
	\int_t^{t+1}\|v\|_{L^\infty(\Omega)}^2ds\leq C\int_{t}^{t+1}\int_\Omega u^2 e^{-v}dxds\leq C,
\end{equation*}which due to Young's inequality indicates that
\begin{equation*}\non
\int_t^{t+1}\|v\|_{L^\infty(\Omega)}ds\leq\int_t^{t+1}\|v\|^2_{L^\infty(\Omega)}+C\leq C.
\end{equation*}
Hence, for any $x\in\Omega$ and $t\geq0$, we obtain that
\begin{equation}\label{unint00}
\int_t^{t+1}v(x,s)ds\leq\int_t^{t+1}\|v\|_{L^\infty(\Omega)}ds\leq C.
\end{equation}
Observing that 
\begin{equation}\non
v_t+ue^{-v}=(I-\Delta)^{-1}[ue^{-v}]\leq (I-\Delta)^{-1}[u]=v,
\end{equation}
 we may fix $x\in\Omega$ and apply Lemma \ref{uniformGronwall} to deduce that
\begin{equation}\non
v(x,t)\leq C\;\;\text{for all}\;t\geq1.
\end{equation}
Since  $C>0$ above is independent of $x$ and
\begin{equation}\non
	v(x,t)\leq v_0(x)e^{e^{-v_*}}\leq ev_0(x)\quad\text{for any}\;x\in\Omega\;\text{and}\;t\in[0,1]
\end{equation}
due to Lemma \ref{pte1}, we conclude that
\begin{equation}\non
	\sup\limits_{t\geq0} v(x,t)\leq C.
\end{equation}This completes the proof.
\end{proof} 

\begin{proof}[Proof of Proposition \ref{prop1}]
Proceeding  along the same lines in the proof of Lemma \ref{lemma_LP}, by Lemma \ref{ublm3} we have \eqref{Lpineq}: 
	\begin{align*}
		\dfrac{d}{dt}\io u^p dx
		+C\io u^{p-2}{|\nabla u|}^2 dx
		\leq
		C \io u^{p+1} dx,
	\end{align*}
where the constant $C>0$ is independent of the time interval $T$.  Noticing that the H\"older's inequality implies
\begin{align*}
\io u^p dx \leq \left( \io u^{p+1} dx\right)^{\frac{p}{p+1}}\cdot |\Omega|^{\frac{1}{p+1}},
\end{align*}
thus
\begin{align*}
|\Omega|^{-\frac{1}{p}} \left( \io u^{p} dx\right)^{\frac{p+1}{p}}
\leq \io u^{p+1},
\end{align*}
and adding $\int_{\Omega}u^{p+1} dx$ both sides, we have
	\begin{align*}
		\dfrac{d}{dt}\io u^p dx
		+|\Omega|^{-\frac{1}{p}} \left( \io u^{p} dx\right)^{\frac{p+1}{p}}
		+C\io u^{p-2}{|\nabla u|}^2 dx
		\leq
		C \io u^{p+1} dx.
	\end{align*}
	Picking $s>0$ sufficiently large in Lemma \ref{fs} and recalling Lemma \ref{ublm1}, we deduce that
	\begin{align*}
			\dfrac{d}{dt}\io u^p \,dx 
	+|\Omega|^{-\frac{1}{p}} \left( \io u^{p} dx\right)^{\frac{p+1}{p}}
		\leq
		C,
	\end{align*}
	and by means of ODE analysis we obtain a uniform-in-time $L^p$-bound for $u$, which concludes the proof.
\end{proof}

\subsection{Unboundedness in Infinite Time}
Stationary solutions $(u,v)$ to \eqref{chemo2} satisfy that
\begin{align*}
\begin{cases}
0=\nabla \cdot ue^{-v} \nabla \left(\log u - v \right)
&\mathrm{in}\ \Omega, \\[1mm]
0=\Delta v -v+u
&\mathrm{in}\ \Omega, \\
u > 0, \ v > 0& \mathrm{in}\ \Omega, \\ 
\displaystyle \frac{\partial u}{\partial \nu} =\frac{\partial v}{\partial \nu}  =0 & \mathrm{on}\ \partial \Omega.
\end{cases}
\end{align*}
Put $\Lambda = \|u\|_{L^1(\Omega)} \in (0,\infty)$. 
In view of the mass conservation and the boundary conditions, 
the above system can be rewritten as the following:
\begin{align} \label{eqn:biharmoniceqn}
\begin{cases}
\displaystyle	 v-\Delta v = \frac{\Lambda}{\int_\Omega e^{v}}
e^{ v} & \mathrm{in}\ \Omega, \\[1mm]
\displaystyle u = \frac{\Lambda}{\int_\Omega e^{ v}}e^{ v}& \mathrm{in}\ \Omega, \\[1mm]
\displaystyle \frac{\partial v}{\partial \nu}  =0 & \mathrm{on}\ \partial \Omega. 
\end{cases}
\end{align}
 Invoking the so-called non-smooth Lojasiewicz--Simon inequality established in \cite{FLP07}, we can prove the following convergence result in the two-dimensional setting. Note that the Lyapunov functional is the same as that for the classical Keller--Segel equation. In addition, the only difference in the dissipation terms is  the extra weighted function $e^{-v}$ in \eqref{Lyapunov},  which is now uniform-in-time bounded from above and below. Thus the proof is  the same as in \cite{FLP07} and we omit the detail here.
\begin{proposition}\label{prop2}
	Let $(u,v)$ be a classical positive solution to problem \eqref{chemo2} in
	$\Omega \times (0,\infty)$. If the solution is uniformly-in-time bounded, 
	there exists a stationary solution $(u_s,v_s)$ to \eqref{eqn:biharmoniceqn} such that 
	\[
	\lim_{t \rightarrow \infty} (u(t),v(t)) = (u_s,v_s) \quad \mbox{{\rm in} }C^2(\overline{\Omega}).
	\]
	as well as
	$$
	E(u_s,v_s) \leq E(u_0,v_0).
	$$
\end{proposition}
\begin{remark}
	The convergence also holds true in higher dimensions provided that the solution is uniform-in-time bounded, see \cite{JZ09,J18}.
\end{remark}


For $\Lambda >0$ put
\[
\mathcal{S}(\Lambda) \triangleq \left\{ (u,v) \in C^2(\overline{\Omega}) :
(u,v ) \mbox{ is a solution to \eqref{eqn:biharmoniceqn}
	satisfying }  \|u\|_{L^1(\Omega)} = \Lambda \right\}.
\]
Here we recall the quantization property of solutions to  \eqref{eqn:biharmoniceqn}.  
By \cite[Theorem 1]{ssAMSA2000} for $\Lambda \not\in 4\pi \mathbb{N}$,  there exists some $C>0$ such that
\[
\sup \{ \|(u,v)\|_{L^\infty (\Omega)} : (u,v) \in
\mathcal{S}(\Lambda) \} \leq C
\]
and
\[
E_\ast(\Lambda) \triangleq \inf \{ E(u,v) : (u,v) \in
\mathcal{S}(\Lambda) \} \geq - C.
\] 
Thus taking account of Proposition \ref{prop2}, 
for a pair of initial data $(u_0,v_0)$ with $v_0=(I-\Delta)^{-1}u_0$ satisfying 
\begin{align}\label{assum0}
\begin{cases}
\|u_0\|_{L^1(\Omega)} = \Lambda\not\in 4\pi \mathbb{N},\\[2pt]
E(u_0,v_0) <E_\ast(\Lambda),
\end{cases}
\end{align}
then the corresponding global solution must blow up in infinite time. Indeed, we have
\begin{lemma} Suppose \eqref{assum0} holds, then the corresponding global solution blows up in infinite time such that
	\begin{equation*}
	\lim\limits_{t\rightarrow+\infty}\|u(t)\|_{L^\infty(\Omega)}=+\infty.
	\end{equation*}
	More precisely, we have
	\begin{equation*}
	\lim\limits_{t\rightarrow+\infty}\int_\Omega uv dx=\lim\limits_{t\rightarrow+\infty}\int_\Omega (|\nabla v|^2+v^2) dx=\lim\limits_{t\rightarrow+\infty}\int_\Omega e^{\alpha v} dx=+\infty
	\end{equation*}for any $\alpha>\frac12.$
\end{lemma}
\begin{proof} The proof is almost the same as that of \cite[Theorem 1]{ssMAA2001} with minor modifications since now the evolution systems are different. We report the detail for reader's convenience. 
	
	First,	according to the convergence result Proposition \ref{prop2}, if $u$ is uniform-in-time bounded, then the global solution must converge to an equilibrium belonging to $\mathcal{S}(\Lambda)$ and thus
	\begin{equation}\label{assum2}
	\lim\limits_{t\rightarrow+\infty}E(u(t),v(t))>-\infty.
	\end{equation}	
It suffices to show that 
	\begin{equation}\label{blowup1}
	\lim\limits_{t\rightarrow+\infty}\int_\Omega uv dx=+\infty,
	\end{equation}	
 since $\int_\Omega uvdx=\int_\Omega (|\nabla v|^2+v^2) dx$ and on the other hand, with \eqref{blowup1} and employing Young's inequality, we have
\begin{align*}
	\alpha\int_\Omega uv\leq& \int_\Omega u\log u+e^{-1}\int_\Omega e^{\alpha v}dx\non\\
	\leq&\frac12\int_\Omega uvdx+E(u,v)+e^{-1}\int_\Omega e^{\alpha v}dx.
\end{align*}
Thus, for any $\alpha>\frac12$,
\begin{equation*}
	\lim\limits_{t\rightarrow+\infty}\int_\Omega e^{\alpha v} dx=+\infty
\end{equation*}
follows from \eqref{blowup1} and \eqref{assum2}.

Now, suppose the contrary: $	\liminf\limits_{t\rightarrow+\infty}\int_\Omega uvdx<+\infty.$ There exist a constant $C_*>0$ and a time sequence $t_k\nearrow+\infty$ such that
	\begin{equation*}
		\int_\Omega u(t_k)v(t_k)dx\leq C_*.
	\end{equation*}
Assumption \eqref{assum2} and \eqref{Lyapunov} indicates that
\begin{equation}\label{int0}
	\int_0^\infty\int_\Omega ue^{-v}|\nabla \log v-\nabla v|^2dxdt<+\infty.
\end{equation}	
	
Since \begin{equation*}
\int_\Omega uv dx=\int_\Omega (|\nabla v|^2+v^2)dx,
\end{equation*}	we derive from \eqref{unb0} that
\begin{equation}\label{ineq0}
\frac{d}{dt}\int_\Omega uvdx\leq 2\int_\Omega uv dx.
\end{equation}
Take $\delta_*>0$ such that $\delta_*(C_*+1)=\frac14$. Then for some $\tilde{t}_k\in(t_k,t_k+\delta_*)$ and any $t_k\leq t<\tilde{t}_k$, we have
\begin{equation*}
	\int_\Omega u(t)v(t)dx<C_*+1
\end{equation*}
and thus \eqref{ineq0} implies that
\begin{equation*}
	\int_\Omega u(\tilde{t}_k)v(\tilde{t}_k)\leq \int_\Omega u(t_k)v(t_k)+2\delta_*(C_*+1)\leq C_*+\frac12.
\end{equation*}
Then since $t\in[t_k,t_k+\delta_*]\mapsto\int_\Omega u(v)v(t)dx$ is continuous, this means that for all $t\in[t_k,t_k+\delta_*]$
\begin{equation}\non
\int_\Omega u(v)v(t)dx\leq C_*+1.
\end{equation}
Note that $\delta_*$ is independent of $k.$ We infer from \eqref{int0} that
\begin{align}
	&\lim\limits_{k\rightarrow+\infty}\inf\limits_{t\in(t_k,t_k+\delta_*)}\int_\Omega ue^{-v}|\nabla \log v-\nabla v|^2dx\non\\
	&\leq\delta_*^{-1}\lim\limits_{k\rightarrow+\infty}\int_{t_k}^{t_k+\delta_*}\int_\Omega ue^{-v}|\nabla \log v-\nabla v|^2dxdt=0.\non
\end{align}
It follows that for some $\hat{t}_k\in(t_k,t_k+\delta_*)$, it holds that
\begin{equation*}
	\lim\limits_{k\rightarrow+\infty}\int_\Omega  u(\hat{t}_k)e^{-v(\hat{t}_k)}|\nabla \log v(\hat{t}_k)-\nabla v(\hat{t}_k)|^2dx=0,
\end{equation*}
which is equivalent to 
\begin{equation*}
		\lim\limits_{k\rightarrow+\infty}\bigg{\|}\nabla \left(u(\hat{t}_k)e^{-v(\hat{t}_k)}\right)^{1/2}\bigg{\|}=0.
\end{equation*}
Since $\|ue^{-v}\|_{L^1(\Omega)}\leq \|u\|_{L^1(\Omega)}=\Lambda$, we deduce by passing to a subsequence that
\begin{equation}\non
	\lim\limits_{k\rightarrow+\infty}\frac{1}{|\Omega|}\int_\Omega\left(u(\hat{t}_k)e^{-v(\hat{t}_k)}\right)^{1/2}dx=C_0
\end{equation}with some constant $C_0\geq0$. Thus, by Poincar\'e--Wirtinger's inequality, we infer that \[\left(u(\hat{t}_k)e^{-v(\hat{t}_k)}\right)^{1/2}\rightarrow C_0\] in $H^1(\Omega)$ and hence for any $p>1$
\begin{equation}\non
	u(\hat{t}_k)e^{-v(\hat{t}_k)}\rightarrow C^2_0\qquad\text{in}\qquad  L^p(\Omega).
\end{equation}

On the other hand, due to relation \eqref{var0a}, we observe by Poincar\'e's inequality that
\begin{align*}
	\|v_t\|\leq& \|(I-\Delta)^{-1}[ue^{-v}]-\overline{ue^{-v}}\|+\|ue^{-v}-\overline{ue^{-v}}\|\\
	\leq& C\|\nabla (ue^{-v})\|_{L^1(\Omega)}\\
	\leq& C\left(\int_\Omega \frac{|\nabla(ue^{-v})|^2}{ue^{-v}}\right)^{1/2}\left(\int_\Omega ue^{-v}\right)^{1/2}\\
	\leq&C\|\nabla (ue^{-v})^{1/2}\|
\end{align*}
which implies that
\begin{equation}\non
	\lim\limits_{k\rightarrow+\infty}\|v_t(\hat{t}_k)\|=0.
\end{equation}

Moreover, applying the Brezis--Merle inequality \cite{BM} to the second equation, we infer with some $\alpha>0$, it holds that
\begin{equation}
	\sup\limits_{k}\int_\Omega e^{\alpha v(\hat{t}_k)}dx<+\infty.\non
\end{equation}

Since
\begin{equation}\non
	\int_\Omega |\nabla v(\hat{t}_k)|^2+v^2(\hat{t}_k)dx=\int_\Omega u(\hat{t}_k)v(\hat{t}_k)dx\leq C_*+1,
\end{equation}
we can now pass to a subsequence such that
\begin{equation}\non
	v(\hat{t}_k)\rightarrow v_\infty\qquad\text{weakly in}\;\;H^1(\Omega)\qquad\text{and}\;\; e^{v(\hat{t}_k)}\rightarrow e^{v_\infty}\qquad\text{strongly in}\;\;L^p(\Omega)
\end{equation}for any $p>1$. The latter convergence follows from an application of the compact embedding $H^1(\Omega)\hookrightarrow L^p(\Omega)$ and Moser-Trudinger's inequality.

Then setting $t=\hat{t}_k$ in the second equation of \eqref{chemo2} and letting $k\rightarrow+\infty,$ we deduce from  that
\begin{equation}\non
	u(\hat{t}_k)=\left(u(\hat{t}_k)e^{-v(\hat{t}_k)}\right)\cdot e^{v(\hat{t}_k)}\rightarrow C_0^2 e^{v_\infty}\quad\text{in}\;\; L^p(\Omega)
\end{equation}for $p>1$. Thus, we obtain that
\begin{equation}\non
	-\Delta v_\infty+v_\infty=C_0^2e^{v_{\infty}}\;\;\text{in}\;\;\Omega,\qquad\frac{\partial v_\infty}{\partial\nu}=0\;\;\text{on}\;\partial\Omega.
\end{equation}
Moreover, $\|v_\infty\|_{L^1(\Omega)}=\|v(\hat{t}_k)\|_{L^1(\Omega)}=\Lambda$ and hence $C_0^2=\Lambda/\int_\Omega e^{v_\infty}dx.$
Letting $u_\infty=\Lambda e^{v_\infty}/\int_\Omega e^{v_\infty}dx$, we obtain that $u(\hat{t}_k)\rightarrow u_\infty$ in $L^p(\Omega)$ which indicates that
\begin{equation}\non
	\lim\limits_{k\rightarrow+\infty}\int_\Omega u(\hat{t}_k)\log u(\hat{t}_k)dx=\int_\Omega u_\infty\log u_\infty dx
\end{equation}as well as
\begin{equation}\non
	\lim\limits_{k\rightarrow+\infty}\int_\Omega u(\hat{t}_k)v(\hat{t}_k)dx=\int_\Omega u_\infty v_\infty dx.
\end{equation}
Therefore, 
\begin{align*}
	E(u_0,v_0)\geq\lim\limits_{k\rightarrow+\infty}E(u(\hat{t}_k),v(\hat{t}_k))=E(u_\infty,v_\infty)\geq E_*(\Lambda),
\end{align*}
which contradicts to our assumption. This completes the proof.
\end{proof}

\subsection{Construction of  Initial Data Satisfying \eqref{assum0}}\label{construction}
This part is devoted to construction of  an example satisfying \eqref{assum0} in the radially symmetric case. Here we give an example in detail based on some calculations in \cite{fs5}.

From now on, we assume $\Omega=B_1(0)$ and we define
for any $\lambda \geq 1$ that
\[
\overline{u}_\lambda (x) := 
\frac{8\lambda^2}{(1+\lambda^2|x|^2)^2}.
\]


	Let $\Lambda \in (8\pi, \infty) \setminus 4\pi \mathbb{N}$. 
	Take $r\in (0,1)$ and for any $r_1 \in
	(0,r)$, let $\phi_{r,r_1}$ be a smooth and radially symmetric function satisfying
	\[
	\phi_{r,r_1}(B_{r_1}(0)) =1, \ 
	0\leq \phi_{r,r_1} \leq 1, \ 
	\phi_{r,r_1}(\R^2 \setminus
	B_r(0)  ) =0,\ x \cdot \nabla \phi_{r,r_1}(x) \leq 0.
	\]

	Now we define $u_0 \triangleq a\overline{u}_\lambda \phi_{r,r_1}
	$ and $v_0 \triangleq (I-\Delta)^{-1}u_0$,
	where $a > \Lambda /8\pi > 1$. We first prove that
	\begin{lemma}
		For any $\lambda>1$ there exists $a> \Lambda /8\pi$  such that
		\begin{eqnarray}\label{mass_inequality}
		&& \io u_0 \,dx= \Lambda.
		\end{eqnarray}
	\end{lemma}
\begin{proof}
	Firstly by changing of variables,  we see that
	\begin{eqnarray} \label{cal0}
		\int_{B_\ell(0)} \overline{u}_\lambda\,dx \nonumber 
		&=& 8 \int_{B_\ell(0)}
		\frac{\lambda^2}{(1+\lambda^2|x|^2)^2}\,dx \nonumber \\
		&=&8 \int_{B_\ell(0)} \frac{dy}{(1+|y|^2)^2}
		\nonumber \\
		&=&16 \pi \int_0^{\lambda \ell} \frac{s}{(1+s^2)^2}\,ds
		\nonumber \\
		&=&8 \pi \int_0^{(\lambda \ell)^2}
		\frac{d\tau}{(1+\tau)^2} \nonumber \\
		&=& 8 \pi \cdot 
		\left( 1 -\frac{1}{1+(\lambda\ell)^2} \right)
		\quad \mbox{ for } \ell >0,
	\end{eqnarray}
	and hence
	\begin{equation}\label{eqn:relat-a-lambda}
	8 \pi \cdot 
	\left( 1 -\frac{1}{1+(\lambda r_1)^2} \right)
	<
	\io \overline{u}_\lambda \phi_{r,r_1}
	<
	8 \pi \cdot 
	\left( 1 -\frac{1}{1+(\lambda r)^2} \right).						 
	\end{equation}
	Then there is a unique constant 
	\begin{equation*}
	a=a(r_1,r,\lambda) > \frac{\Lambda }{8\pi}, 
	\end{equation*}
	satisfying (\ref{mass_inequality}). 
\end{proof}

Observing that $$f(\lambda) \triangleq 1 - \frac{1}{1+(\lambda
	r_1)^2} \to 1 \quad \mbox{ as } \lambda \to \infty,$$
and that
$$
f'(\lambda)=  \frac{2\lambda r_1}{(1+(\lambda r_1)^2)^2}>0 \quad \mbox{for }\lambda>0,
$$ 
we have $1>f(\lambda) \geq f(1)$ for all $\lambda \geq 1$.	
	Thus the constant $a=a(r_1,r,\lambda)$ satisfies
	\begin{equation}\label{bound_of_a}
	 \frac{\Lambda }{8\pi} <a < 
	\frac{\Lambda}{8\pi f(\lambda)}   
	\leq \frac{\Lambda}{8\pi f(1)}. 
	\end{equation}

Now we aim to show that $E(u_0,v_0)$ can be sufficiently negative as $\lambda\rightarrow+\infty$. First, we have
\begin{lemma}
	\label{lmbl1}There exists $C>0$ such that
		\begin{eqnarray*}
	\int_\Omega u_0 \log u_0 \,dx& \leq & 
	 16a\pi \cdot \log \lambda + C \ \mbox{ as }
	\ \lambda \rightarrow \infty,   
	\end{eqnarray*}
	where the constant $C>0$ is independent of $a$.
\end{lemma}
\begin{proof}
	First, we note that
	\begin{eqnarray*}
		\int_\Omega u_0 \log u_0 \,dx
		& \leq & a \io
		\overline{u}_\lambda \log \overline{u}_\lambda + a \log a
		\io \overline{u}_\lambda.
	\end{eqnarray*}
	Since $\log \overline{u}_\lambda \leq \log (8\lambda^2) = 2 \log \lambda +\log 8$ and $\io \overline{u}_\lambda \leq 8\pi$,  
	\begin{eqnarray*}
	\int_\Omega u_0 \log u_0 \,dx& \leq & 
	2a \cdot 8\pi \cdot \log \lambda + C \ \mbox{ as }
	\ \lambda \rightarrow \infty,   
	\end{eqnarray*}
	where we remark that the constant $C$ is independent of $a$ in view of \eqref{bound_of_a}.
\end{proof}	
Since
\begin{equation*}
	\begin{cases}\non
	v_0-\Delta v_0=u_0\;\;\text{in}\; B_1(0),\\
	\partial_\nu v_0=0\;\;\text{on}\;\partial B_1(0),
	\end{cases}
\end{equation*}by introducing $\xi=|x|$, we have
\begin{equation*}
\begin{cases}
v_0(\xi)- v_{0\xi\xi}-\frac{v_{0\xi}}{\xi}=u_0(\xi)\;\;\text{for}\; 0<\xi<1,\\
v_{0\xi}(1)=0,
\end{cases}
\end{equation*}and due to the radially symmetry $v_{0\xi}(0)=0.$ It follows that
\begin{equation*}\non
-\xi^{-1}\left(\xi\frac{\partial}{\partial \xi}v_0(\xi)\right)_\xi=u_0(\xi)-v_0(\xi),
\end{equation*}
from which we obtain that
\begin{equation*}
v_{0\xi}(\xi)=\frac{1}{\xi}\int_0^\xi \left(v_0(\sigma)-u_0(\sigma)\right)\sigma d\sigma
\end{equation*}	
and hence
\begin{equation}\label{v_form}
v_0(\xi)=v_0(1)+\int_\xi^1s^{-1}\int_0^s\sigma u_0(\sigma)d\sigma ds-\int_\xi^1s^{-1}\int_0^s\sigma v_0(\sigma)d\sigma ds.
\end{equation}
Note that	
\begin{align*}
\Lambda=&\int_\Omega v_0dx\non\\
=&2\pi\int_0^1 \xi v_0(\xi)d\xi\non\\
=&2\pi\int_0^1\xi\left[v_0(1)+\int_\xi^1s^{-1}\int_0^s\sigma u_0(\sigma)d\sigma ds-\int_\xi^1s^{-1}\int_0^s\sigma v_0(\sigma)d\sigma ds\right]d\xi\non\\
=&\pi v_0(1)+2\pi\int_0^1\xi\left[\int_\xi^1s^{-1}\int_0^s\sigma u_0(\sigma)d\sigma ds-\int_\xi^1s^{-1}\int_0^s\sigma v_0(\sigma)d\sigma ds\right]d\xi,
\end{align*}
thus	
\begin{align*}
 v_0(1) =\frac{\Lambda}{\pi} -2\int_0^1\xi\left[\int_\xi^1s^{-1}\int_0^s\sigma u_0(\sigma)d\sigma ds-\int_\xi^1s^{-1}\int_0^s\sigma v_0(\sigma)d\sigma ds\right]d\xi.
\end{align*}
Thus by a direct substitution of $v_0(1)$ into \eqref{v_form}, we obtain that
\begin{align}\label{vform}
v_0(\xi)=&\frac{\Lambda}{\pi }-2\int_0^1t\left[\int_t^1s^{-1}\int_0^s\sigma u_0(\sigma)d\sigma ds-\int_t^1s^{-1}\int_0^s\sigma v_0(\sigma)d\sigma ds\right]dt\non\\
&+\int_\xi^1s^{-1}\int_0^s\sigma u_0(\sigma)d\sigma ds-\int_\xi^1s^{-1}\int_0^s\sigma v_0(\sigma)d\sigma ds.
\end{align}
With above representation formula of $v_0$, we obtain the following estimate.
\begin{lemma}\label{lmbl2}There exists $C>0$ such that
\begin{equation*}
	\int_\Omega u_0 v_0 dx\geq 32\pi a^2\log\lambda-C,
\end{equation*}
where $C>0$ is independent of $a$.
\end{lemma}
\begin{proof}Using formula \eqref{vform}, we infer that
\begin{align*}
\int_\Omega u_0 v_0dx=&\frac{\Lambda^2}{\pi }-4\pi\int_0^1u_0(\xi)\left(\int_0^1t\int_t^1s^{-1}\int_0^s\sigma u_0(\sigma)d\sigma dsdt\right)\xi d\xi\non\\
&+4\pi\int_0^1u_0(\xi)\left(\int_0^1t\int_t^1s^{-1}\int_0^s\sigma v_0(\sigma)d\sigma dsdt\right)\xi d\xi\non\\
&+2\pi\int_0^1u_0(\xi)\left(\int_\xi^1s^{-1}\int_0^s\sigma u_0(\sigma)d\sigma ds\right)\xi d\xi\non\\
&-2\pi\int_0^1u_0(\xi)\left(\int_\xi^1s^{-1}\int_0^s\sigma v_0(\sigma)d\sigma ds\right)\xi d\xi\non\\
\triangleq&\frac{\Lambda^2}{\pi }-I_1+I_2+I_3-I_4.
\end{align*}	
In the sequel, we estimate $I_i$ $(i=1,2,3,4)$ separately. First using the fact $\Lambda=\|u_0\|_{L^1(\Omega)}$, we infer that
\begin{align*}
\int_0^1t\int_t^1s^{-1}\int_0^s\sigma u_0(\sigma)d\sigma dsdt\leq & \int_0^1t\int_t^1s^{-1}\int_0^1\sigma u_0(\sigma)d\sigma dsdt\\
=&\frac{\Lambda}{2\pi}\int_0^1t\int_t^1s^{-1}dsdt\\
=&\frac{\Lambda}{8\pi}.
\end{align*}
It follows that
\begin{equation*}
	I_1=4\pi\int_0^1u_0(\xi)\left(\int_0^1t\int_t^1s^{-1}\int_0^s\sigma u_0(\sigma)d\sigma dsdt\right)\xi d\xi\leq \frac{\Lambda}{2}\int_0^1u_0(\xi)\xi d\xi\leq\frac{\Lambda^2}{4\pi}.
\end{equation*}
Secondly, $I_2\geq0$ and thus can be neglected. Next, we proceed with estimate for $I_3$ from below. Denoting \begin{equation}\non
w(\xi)=\int_\xi^1s^{-1}\int_0^s\sigma u_0(\sigma)d\sigma ds,
\end{equation}
then one verifies that	\[w(1)=0\]	and
\begin{equation}\label{ub3}
w_{\xi}(\xi)=-\xi^{-1}\int_0^\xi\sigma u_0(\sigma)d\sigma,
\end{equation}which implies that
\begin{equation}\non
-\Delta w=u_0(x)\;\;\text{in}\;\Omega.
\end{equation}
As a result,	
\begin{align}
I_3=\int_{\Omega}wu_0 dx=-\int_{\Omega}\Delta w wdx=\int_{\Omega}|\nabla w|^2dx.\non
\end{align}
Thus, by \eqref{ub3} we infer that
\begin{align*}
I_3=&\int_{\Omega}|\nabla w|^2dx=2\pi\int_0^1w_{\xi}^2\xi d\xi\non\\
=&2\pi\int_0^1\xi^{-1}\left|\int_0^\xi\sigma u_0(\sigma)d\sigma\right|^2d\xi\non\\
=&2\pi\log\xi\left|\int_0^\xi\sigma u_0(\sigma)d\sigma\right|^2\Bigg{|}_{0}^1+4\pi\int_0^1\log\xi^{-1}\left(\int_0^\xi\sigma u_0(\sigma)d\sigma\right)\xi u_0(\xi)d\xi\non\\
=&4\pi\int_0^1\left(\int_0^\xi\sigma u_0(\sigma)d\sigma\right)\xi\log\xi^{-1} u_0(\xi)d\xi\non\\
\geq&4\pi\int_0^{r_1}\xi\log\xi^{-1}u_0(\xi)\left(\int_0^\xi\sigma u_0(\sigma)d\sigma\right)d\xi,
\end{align*}	
where we used
\begin{equation*}
\lim\limits_{\xi\rightarrow0}\log\xi\left|\int_0^\xi\sigma u_0(\sigma)d\sigma\right|^2=0.
\end{equation*}
Recalling \eqref{cal0}, we deduce that for any $\xi\leq r_1$
\begin{align*}
	\int_0^\xi\sigma u_0(\sigma)d\sigma=\frac{a}{2\pi}\int_{B_{\xi}(0)}\overline{u}_\lambda dx=4a\left(1-\frac{1}{1+\lambda^2\xi^2}\right)
\end{align*}and thus
\begin{align}\label{I_3form}
I_3\geq&4\pi\int_0^{r_1}\log\xi^{-1}\left(\int_0^\xi\sigma u_0(\sigma)d\sigma\right)\xi u_0(\xi)d\xi\non\\
=&16\pi a^2\int_0^{r_1}\xi\log\xi^{-1}\left(1-\frac{1}{1+\lambda^2\xi^2}\right)\frac{8\lambda^2}{(1+\lambda^2\xi^2)^2}d\xi\non\\
=&16\pi a^2\int_0^{r_1}\xi\log\xi^{-1}\frac{8\lambda^2}{(1+\lambda^2\xi^2)^2}d\xi
-16\pi a^2\int_0^{r_1}\xi\log\xi^{-1}\frac{8\lambda^2}{(1+\lambda^2\xi^2)^3}d\xi.
\end{align}	
Since
\begin{equation*}
\int_0^s\frac{\log \tau}{(1+\tau)^2}d\tau=\frac{s\log s-(1+s)\log(1+s)}{1+s}\rightarrow0\;\;\text{as}\;s\rightarrow+\infty,
\end{equation*}
it follows that
\begin{align}\label{I_3_1est}
&\int_0^{r_1}\xi\log\xi^{-1}\frac{8\lambda^2}{(1+\lambda^2\xi^2)^2}d\xi
\non\\
=&-2\int_0^{\lambda^2 r_1^2}\frac{\log s}{(1+s)^2}ds+4\log\lambda\int_0^{\lambda^2r_1^2}\frac{ds}{(1+s)^2}\non\\
=&\frac{2(1+\lambda^2r_1^2)\log(1+\lambda^2r_1^2)-2\lambda^2r_1^2\log \lambda^2r_1^2}{1+\lambda^2r_1^2}+4\log\lambda(1-\frac{1}{1+\lambda^2r_1^2})\non\\
\geq& 4 \log\lambda -C, 
\end{align}
with some $C>0$.
Similarly,
\begin{align*}
&\int_0^{r_1}\xi\log\xi^{-1}\frac{8\lambda^2}{(1+\lambda^2\xi^2)^3}d\xi\non\\
=&-2\int_0^{\lambda^2 r_1^2}\frac{\log s}{(1+s)^3}ds+4\log\lambda\int_0^{\lambda^2r_1^2}\frac{ds}{(1+s)^3},\end{align*}
where for sufficiently large $\lambda$ satisfying $\lambda r_1>1$ we have
\begin{align*}-2\int_0^{\lambda^2 r_1^2}\frac{\log s}{(1+s)^3}ds=&2\int_0^1 \frac{-\log s}{(1+s)^3}ds-2\int_1^{\lambda^2r_1^2}\frac{\log s}{(1+s)^3}ds\non\\
\leq&-2\int_0^1\log sds=2
\end{align*}	
and on the other hand,
\begin{equation*}
	4\log\lambda\int_0^{\lambda^2r_1^2}\frac{ds}{(1+s)^3}=2\log\lambda\left(1-\frac{1}{(1+\lambda^2r_1^2)^2}\right),
\end{equation*}
thus we have some $C>0$ satisfying
\begin{equation}\label{I_3_2est}
\int_0^{r_1}\xi\log\xi^{-1}\frac{8\lambda^2}{(1+\lambda^2\xi^2)^3}d\xi \leq 2 \log \lambda +C.
\end{equation}
Therefore, by \eqref{I_3form}, \eqref{I_3_1est} and \eqref{I_3_2est},  we obtain that as $\lambda\rightarrow+\infty,$
\begin{align*}
I_3\geq 32\pi a^2\log \lambda-C
\end{align*}
with some $C>0$ which is independent of $a$ in view of \eqref{bound_of_a}.

Last for $I_4$, we first observe by H\"older's inequality that
\begin{align*}
\int_\xi^1s^{-1}\int_0^s\sigma v_0(\sigma)d\sigma ds
\leq&\int_\xi^1s^{-1}\left(\int_0^s\sigma v_0^2(\sigma)d\sigma\right)^{1/2}\left(\int_0^s\sigma d\sigma\right)^{1/2}
\non\\
\leq&\int_\xi^1s^{-1}\left(\int_0^1\sigma v_0^2(\sigma)d\sigma\right)^{1/2}\left(\int_0^s\sigma d\sigma\right)^{1/2}
\non\\
=&\frac{\|v_0\|_{L^2(\Omega)}}{2\sqrt{\pi}}\int_\xi^1ds\leq \frac{\|v_0\|_{L^2(\Omega)}}{2\sqrt{\pi}}\leq  \frac{C\Lambda}{2\sqrt{\pi}}
\end{align*}	with $C>0$ depending only on $\Omega$  due to the regularity estimates for elliptic equations \eqref{unibb3}.
Thus,
\begin{align*}
I_4=2\pi\int_0^1u_0(\xi)\left(\int_\xi^1s^{-1}\int_0^s\sigma v_0(\sigma)d\sigma ds\right)dx\leq\frac{C\Lambda^2}{2\sqrt{\pi}}.
\end{align*}
Now, the proof is complete by collecting all above estimates.	
\end{proof}

\begin{proof}[Proof of Theorem \ref{TH2}]
Boundedness of classical solutions with subcritical mass is established in Proposition \ref{prop1}. 
We focus on the supercritical case.
Thanks to Lemma \ref{lmbl1} and \ref{lmbl2}, we infer that for $r \in (0,1)$ and $r_1 \in (0,r)$ there exists some $C=C(r, r_1, \phi_{r,r_1},\Lambda,\Omega)$ such that
	\begin{eqnarray}\label{lyapunov_decreasing}
	\nn
E(u_0,v_0)&=&\int_\Omega u_0\log u_0-\frac12 u_0v_0\\
	& \leq & 
	16 \pi a \log \lambda 
	- 16 \pi a^2 \log \lambda
		+C \non\\
	\nn
	& = & -16 \pi a (1-a)\log \lambda +C\\
	& \leq & - 2\Lambda  \left(\frac{\Lambda}{8\pi} -1\right)
	\log \lambda + C  \rightarrow
	-\infty \ \mbox{ as } \ \lambda \rightarrow \infty,
	\end{eqnarray}
	where we recalled that \eqref{bound_of_a} implies
	$$a(a-1)> 
	\frac{\Lambda }{8\pi} \left(\frac{\Lambda}{8\pi} -1\right).$$
	
	In the last step,  we construct a suitable initial data based on the above observations. 
	For $\Lambda \in (8\pi, \infty) \setminus  4\pi \mathbb{N}$, we first fix $0<r_1 < r$ 
	and function $\phi_{r,r_1}$. 
	Secondly in view of \eqref{lyapunov_decreasing} 
	we can choose some $\lambda > 1$ such that 
	$$
	- 2\Lambda  \left(\frac{\Lambda}{8\pi} -1\right)
	\log \lambda + C < E_* (\Lambda),
	$$
	where $C=C(r, r_1, \phi_{r,r_1},\Lambda,\Omega)$  is the constant in \eqref{lyapunov_decreasing}.  
	Finally we choose $a$ satisfying \eqref{mass_inequality}.
	Therefore by the above discussion $(u_0,v_0)$ satisfies \eqref{assum0}.
\end{proof}

\bigskip
\noindent\textbf{Acknowledgments} \\
K. Fujie is supported by Japan Society for the Promotion of Science (Grant-in-Aid for Early-Career Scientists; No.\ 19K14576).


\begin{thebibliography}{99}
	\itemsep=0pt
	
\bibitem{Ali} N. Alikakos, An application of the invariance principle to reaction-diffusion equations, J. Diff. Equ., \text{33} (1979), 203--225.
	
\bibitem{Anh19} J. Ahn and C. Yoon,
Global well-posedness and stability of constant equilibria in parabolic-elliptic chemotaxis systems without gradient sensing,
Nonlinearity, \text{32} (2019), 1327--1351.

\bibitem{BBTW15}
N. Bellomo, A. Belouquid, Y. Tao and M. Winkler,
Toward a mathematical theory of Keller--Segel models of pattern formation in biology tissues,
Math. Mod. Meth. Appl. Sci., \textbf{25} (2015), 1663--1763.

\bibitem{BCM10}
A. Blanchet, J.A. Carrillo and N. Masmoudi,
Infinite time aggregation for the critical Patlak-Keller-Segel model in $\mathbb{R}^2$,
Commun. Pure Appl. Math., \textbf{61} (2008), 1449--1481.


\bibitem{BM}
H. Br\'ezis and F. Merle,
Uniform estimates and blow-up behavior for solutions of $-\Delta u=V(x)e^{u}$ in two dimensions, Comm. Partial Differential Equations, \textbf{16} (1991), 1223--1253.

\bibitem{BS}
H. Br\'ezis and W. Strauss,
Semi-linear second-order elliptic equations in $L^1$,
 J.\ Math.\ Soc.\ Japan, {\bf 25} (1973), 565--590.
%

\bibitem{CS}
T. Cie\'slak and C. Stinner,
New critical exponents in a fully parabolic quasilinear Keller-Segel system and applications to volume filling models, 
J. Differential Equations, \textbf{258} (2015), 2080--2113.

\bibitem{FLP07}
E. Feireisl, Ph. Lauren\c cot and H. Petzeltov\'a,
On convergence to equilibria for the Keller--Segel chemotaxis model,
J. Different. Equ., \textbf{236} (2007), 551--569.

\bibitem{PRL12} X. Fu, L.H. Huang, C. Liu, J.D. Huang, T. Hwa and P. Lenz,
Stripe formation in bacterial systems with density-suppressed motility,
Phys. Rev. Lett., \textbf{108} (2012), 198102.


\bibitem{FWY2014}
K. Fujie, M. Winkler and T. Yokota,
Blow-up prevention by a logistic sources in a parabolic-elliptic Keller--Segel system with singular sensitivity,
Nonlinear Anal., \textbf{109} (2014), 56--71.

\bibitem{FS2016}
K. Fujie and T. Senba,
Global existence and boundedness of radial solutions to a two dimensional fully parabolic chemotaxis system with general sensitivity,
Nonlinearity, \textbf{29} (2016), 2417--2450.

\bibitem{fs5}
K. Fujie, T. Senba,
 Blowup of solutions to a two-chemical substances chemotaxis system in the critical dimension.
 J.\ Differential Equations, {\bf 266} (2019), 942--976.
 
\bibitem{GM18} 
T. Ghoul and N. Masmoudi,
Minimal mass blowup solutions for the Patlak-Keller-Segel equation,
Commun. Pure Appl. Math., \textbf{71} (2018), 1957--2015.


\bibitem{HW01}
D. Horstmann and G.-F. Wang, Blow-up in a chemotaxis model without symmetry assumptions, 
Euro. J. Appl. Math., \textbf{12} (2001), 159--177.

\bibitem{jaeger_luckhaus}
 W. J\"ager, S. Luckhaus,
On explosions of solutions to a system of partial differential equations modelling chemotaxis,
\rm Trans.\ Amer.\ Math.\ Soc.,\  {\bf 329} (1992), 819--824. 


\bibitem{JZ09}
J. Jiang and Y. Zhang,
On convergence to equilibria for a chemotaxis model with volume-filling effect,
Asympt. Anal., \textbf{65} (2009), 79--102.

\bibitem{J18}
J. Jiang,
Convergence to equilibria of global solutions to a degenerate quasilinear Keller--Segel
system,
Z. Angew. Math. Phys., \textbf{69} (2018):130.

\bibitem{JKW18} H.Y. Jin, Y.J. Kim and Z.A. Wang,
Boundedness, stabilization, and pattern formation driven by density-suppressed motility,
SIAM J. Appl. Math., \text{78} (2018), 1632--1657.


\bibitem{Lankeit}
J. Lankeit,
Infinite time blow-up of many solutions to a general quasilinear parabolic-elliptic Keller--Segel system,
Discrete Contin. Dyn. Syst. Ser. S, Volume 13, Number 2, April (2020), 233--255.

\bibitem{Sciencs11} C. Liu et al.,
Sequential establishment of stripe patterns in an expanding cell population, 
Science, \textbf{334} (2011), 238.

\bibitem{Nagai97}T. Nagai, T. Senba and K. Yoshida,
Application of the Trudinger--Moser inequality to a parabolic
system of chemotaxis,
Funkcialaj Ekvacioj, \textbf{40} (1997), 411--433.

\bibitem{Nagai98}
T. Nagai and T. Senba, Global existence and blow-up of radial solutions to a parabolic-elliptic system of chemotaxis, Adv. Math. Sci. Appl., \textbf{8}(1998), 145--156. 


\bibitem{Nagai00}
T. Nagai, T. Senba and T. Suzuki, Chemotactic collapse in a parabolic system of mathematical biology, Hiroshima Math. J., \textbf{30} (2000), 463--497.


\bibitem{nagai2001}
 T. Nagai,
Blowup of nonradial solutions to parabolic-elliptic systems modeling chemotaxis in two-dimensional domains,
\rm J. Inequal. Appl., {\bf 6} (2001), 37--55.

\bibitem{ssAMSA2000}
T. Senba and T. Suzuki,
Some structures of the solution set for a stationary system of chemotaxis,
\rm Adv.\ Math.\ Sci.\ Appl.,\ {\bf 10} (2000), 191--224.

\bibitem{ssMAA2001}
 T. Senba and T. Suzuki,
Parabolic system of chemotaxis: blowup in a finite and the infinite time,
\rm  Methods Appl.\ Anal.,\ {\bf 8} (2001), 349--367. 

\bibitem{TaoWincritical} Y.S. Tao and M. Winkler,
Critical mass for infinite-time aggregation in a chemotaxis model with indirect signal production, 
J. Eur. Math. Soc. (JEMS), \textbf{19} (2017), 3641--3678.

\bibitem{TaoWin17} Y.S. Tao and M. Winkler,
Effects of signal-dependent motilities in a Keller--Segel-type reaction-diffusion system,
Math. Mod. Meth. Appl. Sci., \text{27} (2017), 1645--1683.

\bibitem{Temam}
R. Temam,
Infinite-dimensional dynamical systems in Mechanics and Physics,
Applied Mathematical Sciences, \textbf{68}, Springer-
Verlag, New York, 1988.

\bibitem{Win10} M. Winkler,
Boundedness in the higher-dimensional parabolic-parabolic chemotaxis system with logistic source, Comm. Partial Differential Equations, \textbf{35} (2010), 1516--1537.


\bibitem{Win13} M. Winkler, Finite-time blow-up in the higher-dimensional parabolic-parabolic Keller--Segel system, J. Math. Pures Appl., \textbf{100} (2013), 748--767.


\bibitem{YK17} C. Yoon and Y.J. Kim,
Global existence and aggregation in a Keller--Segel model with Fokker--Planck diffusion,
Acta Appl. Math., \textbf{149} (2017), 101--123.







	

\end{thebibliography}
\end{document}